\documentclass[11pt]{amsart}

\usepackage[utf8]{inputenc}
\usepackage{hyperref}
\usepackage{amsfonts}
\usepackage{amsmath}
\usepackage{amsthm}
\usepackage{float}
\usepackage{graphicx}

\usepackage{subfig}
\usepackage{biblatex}
\usepackage{witharrows}
\usepackage{mathtools}
\usepackage{comment}
\usepackage{caption}
\usepackage{geometry}
\usepackage{etoolbox}
\title[Rotation sets and entropy]{Rotation Sets and Topological Entropy for Random Circle Endomorphisms}

\author[Z. Li]{Zixu Li}
\address{School of Mathematics, Sun Yat-sen University, Zhuhai, China}
\email{zixu.li915@gmail.com}

\author[S. Lloyd]{Simon Lloyd}
\address{School of Mathematics and Physics, Xi'an Jiaotong-Liverpool University, Suzhou, China}
\email{Simon.Lloyd@xjtlu.edu.cn}

\author[S. Roma\~na]{Sergio Roma\~na}
\address{School of Mathematics, Sun Yat-sen University, Zhuhai, China}
\email{sergio@mail.sysu.edu.cn}

\numberwithin{equation}{section}
\newtheorem{theorem}{Theorem}[section]
\newtheorem{lemma}{Lemma}[section]
\newtheorem{proposition}{Proposition}[section]
\newtheorem{corollary}{Corollary}[section]
\newtheorem*{Q1}{Question~1} 
\newtheorem*{Q2}{Question~2} 
\theoremstyle{definition}
\newtheorem{definition}{Definition}[section]
\theoremstyle{remark}
\newtheorem{remark}{Remark}[section]
\newtheorem{example}{Example}[section]

\begin{document}

\begin{abstract}
We study the topological dynamics of random circle endomorphisms of degree one over an ergodic measure-preserving dynamical system. Under an integrability assumption, we prove that the random rotation set is almost surely the compact interval whose endpoints are the mean random rotation numbers of the associated lower and upper random maps.
We also show that the natural orbitwise versions of the random rotation set agree almost surely, and on the same full-measure set, every value in this interval is realised as the asymptotic average displacement along an individual orbit. In addition, every closed subinterval of the random rotation set is realised, on a full-measure set, as the set of accumulation values of the displacement averages along a single orbit. Finally, we prove that a positive length of the random rotation set implies a positive random topological entropy, in contrast to random monotone maps, which have zero random topological entropy. We illustrate the theory by computing the random rotation set and random topological entropy for a piecewise linear example.
\end{abstract}

\maketitle

\section{Introduction}
Rotation theory is one of the central tools in the study of one-dimensional dynamics. The concept of rotation number originates from Poincar\'e~\cite{Poincare1885}. For a degree one circle homeomorphism, the rotation number quantifies the average angular displacement along an orbit and can be used to characterise periodicity in the dynamics. For a continuous degree one circle endomorphism with no assumption of injectivity, Newhouse, Palis and Takens~\cite{Newhouseetal1983} showed that the average displacement along an orbit does not necessarily converge to a single rotation number, but the set of accumulation points of the average displacement along all orbits instead defines an interval called the rotation set. Ito~\cite{Ito1981} proved that the set of accumulation points of the average displacements along orbits is closed. Bam\'on, Malta, Pacifico and Takens~\cite{Bamonetal1984} showed that any subinterval of the rotation set can be realised as the set of accumulation points of the average displacement along a single orbit. A related orbitwise viewpoint was later developed by Alsed\`a, Ma\~nosas and Chas~\cite{Alseda2002}, who studied rotation sets associated with individual orbits and with orbit closures. The non-triviality of the rotation set is also related to the global complexity of the dynamics, particularly the appearance of positive topological entropy, as demonstrated by Misiurewicz~\cite{Misiurewicz1982}. Sharp lower bounds for the entropy in terms of the rotation set were later obtained by Alsed\`a, Llibre, Ma\~nosas and Misiurewicz~\cite{Alseda_1988}.

Rotation theory has also been developed for random systems, consisting of compositions of fibre maps composed according to the orbits of a measure-preserving base transformation. Analogues of the rotation number for random systems were first developed for order-preserving fibre maps. The extension of the rotation number to random circle homeomorphisms is due to Ruffino~\cite{Ruffino2000}, who proved its existence and independence of the initial fibre point. Li and Lu~\cite{LiLu2008} extended the theory to integrable random monotone lifts and proved that the random rotation number varies continuously with respect to the lift. For random continuous maps of the circle of degree one, the asymptotic average displacement depends on the initial fibre point, naturally leading to a set-valued invariant, called the random rotation set. Random rotation sets have been studied by Zhu~\cite{Zhu2023} for random maps homotopic to the identity and with bounded generators on $n$-dimensional tori.
In this paper, we analyse the structure of the random rotation set for random circle endomorphisms of degree one, the equivalence of several definitions and the realisation of values and subintervals of the random rotation set as the set of accumulation points of the average displacement along individual orbits.

A corresponding picture was developed by Glendinning, J\"ager and Stark~\cite{Glendinningetal2009} in the uniquely ergodically-forced setting, in which the base transformation is assumed to be a continuous and uniquely ergodic transformation of a compact space and the dependence of the lifts of the fibre maps on the base point is assumed to be continuous. An important special case is quasiperiodically-forced maps, for which the base transformation is an irrational rotation of the circle. They proved that the (fibred) rotation set is a compact interval and that every value is realised on a minimal set. We have weaker assumptions in this paper: we assume that the base system is an ergodic measure-preserving transformation and the dependence of the random lift on the base point is integrable. Hence, our realisation results are formulated for individual random orbits on a full-measure set, rather than for minimal sets of a compact skew product system.

Motivated by these works, we develop a random analogue of the classical rotation set theory for degree one circle endomorphisms. 
We develop a rotation theory for degree one random circle endomorphisms over an ergodic base transformation under the assumption that the generator of the random lifts is integrable. We introduce the associated random rotation set, defined as the topological closure of the set of upper limits of the asymptotic displacements along all orbits. Our first main result (see Theorem \ref{thm:rotset}) is that the random rotation set agrees almost surely with the compact interval whose endpoints are given by the mean random rotation numbers of the lower and upper random maps. We also prove that the closure in the definition is unnecessary, which generalises the result of Ito~\cite{Ito1981} to the random setting. Moreover, we show that, on a common full-measure set, several alternate formulations also produce the same set: the upper limit, the lower limit, ordinary and subsequential limit definitions all give rise to the same interval as the random rotation set. We also prove that every value in the interval is realised as the actual asymptotic displacement of an orbit of the original random endomorphism, and every closed subinterval is realised as the full set of accumulation points of the average displacement along a single orbit (see Corollary \ref{cor:realsubint}). This is a random analogue of the deterministic result of Bam\'on et al.~\cite{Bamonetal1984}. 

Lastly, we consider how the complexity of the dynamics of random circle endomorphisms relates to the rotation theory. For random circle endomorphisms with bounded generator, Zhu~\cite{Zhu2023} shows that the rotational entropy is zero, a concept first introduced by Botelho~\cite{Botelho1991} in the context of annulus maps. Here we consider the concept of random topological entropy (see \cite{Bogenschutz1993, Kifer1986}). Our final main result is that a non-degenerate random rotation set implies positive random topological entropy (see Theorem \ref{thm:positivetopentropy}). In contrast, for random monotone circle maps, the random topological entropy vanishes.

The proof of Theorem \ref{thm:rotset} uses a family of random monotone maps that interpolate between the lower and upper monotone maps. The key step is a coincidence argument for orbits of the interpolating family and orbits of the endomorphism, which shows that every value in the random rotation set can be realised by a random orbit of the original random circle endomorphism. The proof of Theorem \ref{thm:positivetopentropy} is obtained by converting the positive length of the random rotation set into finite-time full-cover blocks that occur with positive frequency along typical orbits of the base system. As an application, we consider a piecewise linear random circle endomorphism over a Bernoulli base transformation and compute explicitly the random rotation set and random topological entropy.

The paper is organised as follows. In Section~\ref{sec:main}, we introduce the basic setting and state the main results. In Section~\ref{sec:randomintnum} we develop the random interpolating construction, prove the interval structure theorem and realisation results for the random rotation set over an ergodic base system and record several basic properties of the random rotation set such as realisation results, change of lifts, monotonicity and cohomological invariance. In Section~\ref{sec:randomtopo} we prove that the positive topological entropy implication in the ergodic case and establish the zero entropy result for random circle monotone maps. The final section contains a brief discussion of the scope of the method and some open questions concerning invariant set realisation and symbolic entropy mechanisms.

\section{Notation and statements of main results}\label{sec:main}

Let $(\Omega,\mathcal{F},\mathbb{P},\sigma)$ be a probability measure-preserving dynamical system, let $\mathbb{S}^1=\mathbb{R}/\mathbb{Z}$ and let $\pi:\mathbb{R}\to \mathbb{S}^1$ denote the natural projection. Let $\mathrm{End}_1(\mathbb{S}^1)$ denote the set of continuous circle endomorphisms of degree one, let $\mathrm{Mon}_1(\mathbb{S}^1)$ denote the set of continuous monotone circle maps of degree one and let $\mathrm{Homeo}_1(\mathbb{S}^1)$ denote the set of degree one homeomorphisms of the circle.

A \emph{random circle endomorphism} over $(\Omega,\mathcal{F},\mathbb{P},\sigma)$ is a measurable map $\phi:\mathbb{N}_0\times\Omega\times \mathbb{S}^1\to\mathbb{S}^1$ with the following properties for each $n,m\in\mathbb{N}_0=\{0,1,2,\ldots\}$, $\omega\in\Omega$ and $x\in\mathbb{S}^1$:
\begin{enumerate}
\item $\phi(0,\omega,\cdot)=\mathrm{Id}_{\mathbb{S}^1}$;
\item $\phi(n,\omega,\cdot)\in \mathrm{End}_1(\mathbb{S}^1)$;
\item $\phi(m+n,\omega,x)=\phi(n,\sigma^m\omega,\phi(m,\omega,x))$.
\end{enumerate}
We define \emph{random monotone circle maps}, and \emph{random circle homeomorphisms} similarly, by replacing $\mathrm{End}_1(\mathbb{S}^1)$ in item 2 with $\mathrm{Mon}_1(\mathbb{S}^1)$ or $\mathrm{Homeo}_1(\mathbb{S}^1)$ respectively.
We denote the map $\phi(1,\omega,\cdot)$ by $f_\omega$ and the \emph{cocycle} $\phi(n,\omega,\cdot)$ by $f^{(n)}_\omega$, and so
\begin{align}\label{eq:cocycle}
\phi(n,\omega,\cdot) = f^{(n)}_\omega \coloneq f_{\sigma^{n-1}\omega}\circ \cdots \circ f_{\sigma\omega} \circ f_\omega.
\end{align}
The maps $f_\omega$ are called \emph{fibre maps} and $\sigma$ is called the \emph{base transformation}.

Let $\mathcal{E}(\Omega)$ denote the collection of measurable maps $f:\Omega\to \mathrm{End}_1(\mathbb{S}^1)$, where $\omega\mapsto f_\omega$. The map $f\in\mathcal{E}(\Omega)$ is called the \emph{generator} of the random circle endomorphism over $(\Omega,\mathcal{F},\mathbb{P},\sigma)$ defined by (\ref{eq:cocycle}).
Similarly, we write $\mathcal{M}(\Omega)\coloneq\{f:\Omega\to \mathrm{Mon}_1(\mathbb{S}^1)\text{ measurable}\}$ and $\mathcal{H}(\Omega)\coloneq\{f:\Omega\to \mathrm{Homeo}_1(\mathbb{S}^1)\text{ measurable}\}$.

 To analyse the asymptotic behaviour of the system, we lift the dynamics from the circle to the real line. Let $\mathrm{End}_1(\mathbb{R})$ denote the set of continuous maps $F:\mathbb{R}\to\mathbb{R}$ satisfying the \emph{degree one property} 
\begin{align}\label{eq:degree1}
F(x+1)=F(x)+1\qquad\textrm{for all }x\in\mathbb{R}.
\end{align}
We also use the notation $\mathrm{Mon}_1(\mathbb{R})=\{F\in\mathrm{End}_1(\mathbb{R}): F\textrm{ is non-decreasing}\}$ and $\mathrm{Homeo}_1(\mathbb{R})=\{F\in\mathrm{End}_1(\mathbb{R}): F\textrm{ is strictly increasing}\}$.
Given $f\in \mathrm{End}_1(\mathbb{S}^1)$, a lift of $f$ is a map $F\in \mathrm{End}_1(\mathbb{R})$ such that $\pi\circ F=f\circ \pi$. Lifts are unique up to an integer. The \emph{displacement} of a lift $F\in \mathrm{End}_1(\mathbb{R})$ is the continuous $1$-periodic function $\Delta F:\mathbb{R}\to\mathbb{R}$ given by $\Delta F(x)=F(x)-x$.

A \emph{random lift} of $f\in \mathcal{E}(\Omega)$ is a measurable map $F:\Omega\to \mathrm{End}_1(\mathbb{R})$
such that $F_\omega$ is a lift of $f_\omega$ for every $\omega\in\Omega$. The associated cocycle is defined by
$F_\omega^{(n)}=F_{\sigma^{n-1}\omega}\circ\cdots\circ F_\omega.$
We define $\widetilde{\mathcal{E}}(\Omega)$ as the collection of measurable maps $F:\Omega\to \mathrm{End}_1(\mathbb{R})$ that satisfy the \emph{integrability condition}
\begin{align}\label{eq:integrability}
\int_{\omega\in\Omega} \|F_\omega-\mathrm{Id}\|\:\mathrm{d}\mathbb{P}(\omega) < \infty,
\end{align}
where $\|\cdot\|$ denotes the uniform norm of the displacement function, and we identify elements of $\widetilde{\mathcal{E}}(\Omega)$ that are equal to $\mathbb{P}$-almost everywhere. 

Let $\widetilde{\mathcal{M}}(\Omega)$ denote the corresponding class of integrable random lifts of elements of $\mathcal{M}(\Omega)$. For $F\in \widetilde{\mathcal{M}}(\Omega)$, Li and Lu \cite{LiLu2008} proved that there exists an integrable function $\rho_F:\Omega\to\mathbb{R}$, called the \emph{random rotation number}, such that for any $x\in\mathbb{R}$,
\begin{equation}\label{eq:LiLu}
\lim_{n\to\infty} \frac{F^{(n)}_\omega(x)-x}{n} = \rho_F(\omega)\qquad\textrm{for a.e. }\omega\in\Omega.
\end{equation}
Moreover, they prove that $\rho_F$ depends continuously on $F$. If the base transformation $\sigma$ is ergodic, then $\rho_F$ is constant $\mathbb{P}$-almost everywhere, in which case
\begin{equation}
\rho_F(\omega)=\rho(F) \coloneq \int_\Omega \rho_F\,\mathrm{d}\mathbb{P}\qquad\textrm{for a.e. }\omega\in\Omega.
\end{equation}

\begin{remark}
In the deterministic case, since lifts of a monotone circle map only differ by an integer, the fractional part of the rotation number is independent of the choice of the lift.
For random monotone circle maps, Ruffino and Rodrigues~\cite{RodriguesRuffino2013} investigated the dependence of the random rotation number on the choice of lift, introducing a parametrized family of $(q,\alpha)$-random rotation numbers, for $(q,\alpha)\in\mathbb{R}^2$, defined in terms of a random lift $F_{(q,\alpha)}\in \widetilde{\mathcal{M}}(\Omega)$ specified by
$$
F_{(q,\alpha),\omega}(q)\in [\alpha,\alpha+1)\qquad\textrm{for all }\omega\in\Omega.
$$ 
Given any $f\in\mathcal{E}(\Omega)$ and $(q,\alpha)\in\mathbb{R}^2$, the random lift $F_{(q,\alpha)}$ satisfies the integrability condition and so $F_{(q,\alpha)}\in\widetilde{\mathcal{E}}(\Omega)$.
In this article, we do not require any specific choice of lift: we work with arbitrary random lifts $F$ satisfying the integrability condition \eqref{eq:integrability}. 
\end{remark}

By analogy with the definition used by Newhouse, Palis, Takens~\cite{Newhouseetal1983} for the rotation set of a single circle endomorphism, we propose the following definition of the random rotation set of a random circle endomorphism.

\begin{definition}\label{def:rotset}
Consider a random circle endomorphism over $(\Omega, \mathcal{F}, \mathbb{P}, \sigma)$ generated by $f\in\mathcal{E}(\Omega)$. The \emph{random rotation set} $\mathrm{Rot}_F$ of a random lift $F\in \widetilde{\mathcal{E}}(\Omega)$ is defined for each $\omega\in\Omega$ to be the closed set
\begin{align}\label{eq:rotset}
\mathrm{Rot}_F(\omega)\coloneq\overline{\left\{ \limsup_{n\to \infty}\frac{F^{(n)}_\omega(x)-x}{n}: x\in \mathbb{R}\right\}}.
\end{align}
\end{definition}

We define the \emph{upper random rotation set} $\mathrm{Rot}^+_F$, the \emph{lower random rotation set} $\mathrm{Rot}^-_F$, the \emph{ordinary random rotation set} $\mathrm{Rot}^{\mathrm{ord}}_F$ and the \emph{subsequential random rotation set} for each $\omega\in\Omega$ by
\begin{align}\label{eq:upperrotset}
\mathrm{Rot}^+_F(\omega)& \coloneq \left\{\limsup_{n\to \infty} \frac{F^{(n)}_\omega(x)-x}{n}: x\in \mathbb{R}\right\},\\
\mathrm{Rot}^-_F(\omega)& \coloneq\left\{\liminf_{n\to \infty} \frac{F^{(n)}_\omega(x)-x}{n}: x\in \mathbb{R}\right\},\\
\mathrm{Rot}^{\mathrm{ord}}_F(\omega)& \coloneq\left\{c\in \mathbb{R}: \exists\  x\in \mathbb{R} \textrm{ such that }\lim_{n\to \infty} \frac{F^{(n)}_\omega(x)-x}{n}=c\right\},\\
\mathrm{Rot}_F^{\mathrm{sub}}(\omega)& \coloneq \left\{
c\in\mathbb R:
\exists\, n_k\to\infty,\ \exists\, y_k\in\mathbb R
\text{ such that }
\lim_{k\to\infty}
\frac{F_\omega^{(n_k)}(y_k)-y_k}{n_k}
=c
\right\}.
\end{align}
See \cite{Alsedaetal1993} for corresponding definitions in the deterministic setting.
For convenience, when $F\in\widetilde{\mathcal{E}}(\Omega)$ is understood, we sometimes denote the average displacement over $n$ steps of the orbit of $x$ by
$$
A_n(\omega,x):=\frac{F_\omega^{(n)}(x)-x}{n}
$$
for $n\geq 1$, $\omega\in \Omega$ and $x\in \mathbb{R}$.

Circle endomorphisms are often studied in terms of their upper and lower maps, and we make use of the equivalent concepts for random circle endomorphisms. The \emph{lower map} $F_L$ and \emph{upper map} $F_U$ of $F\in \widetilde{\mathcal{E}}(\Omega)$ are defined as follows:
\begin{align}\label{eq:envelopes}
F_{\omega, L}(x) \coloneq \inf_{y \ge x} F_\omega(y)
\quad\textrm{and}\quad
F_{\omega, U}(x) \coloneq \sup_{y \le x} F_\omega(y).
\end{align}
Both $F_{L}$ and $F_{U}$ are monotonically increasing and degree one random monotone maps, so $F_L, F_U \in \widetilde{\mathcal{M}}(\Omega)$. By \eqref{eq:LiLu}, these random monotone maps admit integrable random rotation numbers, which we denote by $\rho_{F_L}$ and $\rho_{F_U}$, respectively. If the base transformation is ergodic, then $\rho_{F_L}$ and $\rho_{F_U}$ are constant almost everywhere, so for almost every $\omega\in\Omega$ the random rotation numbers $\rho_{F_L}(\omega)$ and $\rho_{F_U}(\omega)$ are equal to their mean values over the base space $\Omega$: that is,
\begin{align}
\rho_{F_L}(\omega)=\rho(F_L)\coloneq\int_\Omega \rho_{F_L}\,\mathrm{d}\mathbb{P}\quad\textrm{and}\quad\rho_{F_U}(\omega)=\rho(F_U)\coloneq\int_\Omega \rho_{F_U}\,\mathrm{d}\mathbb{P}.
\end{align}

Our first main result describes the structure of the random rotation set for any integrable lift $F\in\widetilde{E}(\Omega)$ of a random circle endomorphism over an ergodic base transformation. We prove that the random rotation set $\mathrm{Rot}_F$ is a random compact set almost surely equal to the closed interval bounded by the mean rotation numbers of the upper and lower maps. In this sense, the result extends the interval structure demonstrated by Glendinning, J\"ager and Stark~\cite{Glendinningetal2009} for uniquely ergodically-forced circle maps to the general ergodic random framework. We also show that the random rotation set is almost everywhere equal to the upper random rotation set, generalising the result of Ito~\cite{Ito1981} to the random setting. Moreover, we show that each value in the random rotation set can be realised by the orbit of some point.

\begin{theorem}[Structure and realisation of the random rotation set]\label{thm:rotset}
Let $F\in \widetilde{\mathcal{E}}(\Omega)$ be an integrable random lift of the generator of a random circle endomorphism over an ergodic probability measure-preserving dynamical system $(\Omega, \mathcal{F}, \mathbb{P}, \sigma)$. Then there exists a $\sigma$-invariant full-measure set $\Omega^*\subset \Omega$ such that for every $\omega\in \Omega^*$, the following hold:
\begin{enumerate}
\item The upper, lower, ordinary and subsequential random rotation sets agree on $\Omega^*$ and constitute an essentially constant compact interval:
\begin{align}
\mathrm{Rot}_F(\omega)=\mathrm{Rot}_F^+(\omega)=\mathrm{Rot}_F^-(\omega)=\mathrm{Rot}^{\mathrm{ord}}_F(\omega)=\mathrm{Rot}^{\mathrm{sub}}_F(\omega)=[\rho(F_L),\rho(F_U)], 
\end{align}
where $F_L$ and $F_U$ denote the lower and upper maps of $F$. In particular, the closure in the definition of $\mathrm{Rot}_F$ is redundant for every $\omega\in \Omega^*$. We denote this almost sure common interval by $\mathrm{Rot}(F)\coloneq [\rho(F_L), \rho(F_U)]$.
\item Each value in this interval is realised as the limit of the average displacement of an actual orbit. More precisely, for every value $c\in \mathrm{Rot}(F)$, there exists a point $x_c=x_c(\omega) \in [0, 1]$ such that
\begin{align}
\lim_{n\to\infty} \frac{F_\omega^{(n)}(x_c(\omega))-x_c(\omega)}{n}=c. 
\end{align}
\end{enumerate}
\end{theorem}

The proof is given in Section~\ref{sec:randomintnum}. We also use an interpolating family and coincidence sets as in \cite{Glendinningetal2009}, but their roles are different in the ergodic setting. The key points are the family of random monotone maps that interpolate between $F_L$ and $F_U$ of the random circle endomorphism $F$, the $L^1$-continuity of the random rotation number for random circle monotone maps and a coincidence set argument which transfers orbit information from the interpolating family to the original non-invertible map. Since uniform convergence over a compact base space is not available, we first construct a common full-measure set on which the limiting average displacement exists for all the random monotone map in the interpolating family. This simultaneous convergence is also used to remove the closure in the definition of the random rotation set and to identify it with the other types of random rotation set.

The realisation of single values in the random rotation set in Theorem~\ref{thm:rotset} can be extended into a statement concerning the realisation of subintervals. The next result, which is proved in Section \ref{sec:realsubint}, is a random analogue of the deterministic subinterval realisation theorem in \cite{Bamonetal1984}.

\begin{theorem}(Realisation of upper and lower asymptotic displacement averages)\label{thm:realupperlower}
Let $F\in \widetilde{\mathcal{E}}(\Omega)$ be an integrable random lift of the generator of a random circle endomorphism over an ergodic probability measure-preserving dynamical system $(\Omega, \mathcal{F}, \mathbb{P}, \sigma)$. For any $a, b\in [\rho(F_L), \rho(F_U)]$ with $a\leq b$, there exists a full-measure set $\Omega_{a, b}\subset \Omega$ such that for every $\omega\in \Omega_{a, b}$ there exists $x_{a, b}(\omega)\in \mathbb{R}$ satisfying 
$$\liminf_{n\to \infty}\frac{F^{(n)}_\omega(x)-x}{n}=a \quad\textrm{and}\quad 
\limsup_{n\to \infty}\frac{F^{(n)}_\omega(x)-x}{n}=b.$$
\end{theorem}

\begin{corollary}(Realisation of closed subintervals)\label{cor:realsubint}
Assume the same hypotheses on the base system and the random lift $F$ as in Theorem~\ref{thm:realupperlower}. Let $I=[a, b]\subset [\rho(F_L), \rho(F_U)]$. Then there exists a full-measure set $\Omega_I \subset \Omega$ such that, for a.e. $\omega \in \Omega_I$, there exists $x_I(\omega)\in \mathbb{R}$ such that
\begin{align*}
\operatorname{Clust}\left\{\frac{F^{(n)}_\omega(x_I(\omega))-x_I(\omega)}{n}: n\geq 1\right\}=I,
\end{align*}
where $\operatorname{Clust}\{u_n: n\geq 1\}$ denotes the set of all limits of subsequences of $u_n$.
\end{corollary}
    
Now we turn to random topological entropy. In the deterministic theory of degree one circle endomorphisms, a non-trivial rotation set signifies topological complexity. In the random setting, deducing positivity of entropy from a positive length random rotation set requires an additional argument to obtain the growth rates of the size of $\varepsilon$-separated sets from the non-uniform asymptotic displacement information provided by the random rotation set.

Firstly, we consider random circle monotone maps. Since these maps are order-preserving, the random rotation set reduces to a singleton and the random topological entropy is zero. 

\begin{proposition}\label{pro:zeroentropy}
Let $(\Omega, \mathcal{F},\mathbb{P},\sigma)$ be a measure-preserving dynamical system and let $f \in \mathcal{M}(\Omega)$ be the generator of a random monotone circle map. Then $f$ has zero random topological entropy.
\end{proposition}

\begin{remark}
The monotone case is quite different from the non-invertible case. For random monotone circle maps, the random rotation set reduces to a singleton, while in the non-invertible case it may have non-trivial length. The positive entropy result proved later therefore belongs to a more general setting.
\end{remark}

The length of the set $\mathrm{Rot}(F)$ measures the spread in the asymptotic displacements across different orbits. Our final main result shows that, in the ergodic case, a positive length of the random rotation set implies a positive random topological entropy.

\begin{theorem}\label{thm:positivetopentropy}
Let $F\in \widetilde{\mathcal{E}}(\Omega)$ be an integrable random lift of the generator of a random circle endomorphism over an ergodic probability measure-preserving dynamical system $(\Omega, \mathcal{F}, \mathbb{P}, \sigma)$. If the random rotation set $\mathrm{Rot}(F)$ is non-degenerate, then $f$ has positive random topological entropy.
\end{theorem}

The proof is given in Section~\ref{sec:randomtopo}. Non-degeneracy of the random rotation set means that $\rho(F_U)>\rho(F_L)$. In the proof, we use the positive length of the random rotation set to construct a positive density family of full-cover blocks and then use a symbolic coding argument on a finite cover of the circle.

\section{Random interpolating maps and random rotation numbers}\label{sec:randomintnum}

To study the random rotation set, we establish an algebraic property for the upper and lower maps of degree one circle endomorphisms, and then extend the deterministic interpolating map construction of Glendinning, J\"ager, and Stark~\cite{Glendinningetal2009} to the measurable setting.

\subsection{Composition of monotone envelopes}
For a deterministic degree-one lift $G\in \mathrm{End}_1(\mathbb{R})$, its lower and upper maps are defined respectively by 
\begin{align}
G_L(x)\coloneq\inf_{y\ge x} G(y)
\quad\textrm{and}\quad 
G_U(x)\coloneq\sup_{y\le x} G(y).
\end{align}
Clearly, both $G_L$ and $G_U$ belong to $\mathrm{Mon}_1(\mathbb{R})$ and $G_L\leq G\leq G_U$. The following lemma shows that taking the upper or lower envelope commutes with function composition.

\begin{lemma}\label{lem:upperlowercomposition}
Let $F, G\in \mathrm{End}_1(\mathbb{R})$. We have
\begin{align}
(G\circ F)_U=G_U \circ F_U \quad\textrm{and}\quad (G\circ F)_L=G_L \circ F_L.
\end{align}
\end{lemma}
\begin{proof}
Given $x\in\mathbb{R}$, since $F$ is continuous on the closed interval $[x-1,x]$, it attains its maximum. So there exists a point $x^*\in [x-1, x]$ such that $F(x^*)=\sup\{F(y): y\in [x-1, x]\}$. Hence we have
\begin{align*}
G_U\circ F_U (x)&=\sup\{G(y): F(x^*)-1\leq y \leq F(x^*)\}
=\sup\{G(y): y\leq F(x^*)\}.
\end{align*}
The upper map of the composite function $G\circ F$ satisfies
\begin{align*}
(G \circ F)_U(x) &=\sup \{ G\circ F(y): y \leq x\}\\
 &=\sup\{ G\circ F(y): x-1 \leq y \leq x\}\\
 &=\sup\left\{G(z): \inf_{y\in [x-1, x]} F(y)\leq z \leq \sup_{y\in [x-1, x]} F(y)\right\}\\
 &=\sup\left\{G(z): \inf_{y\in [x-1, x]} F(y)\leq z \leq F(x^*)\right\},
\end{align*}
by the extreme value theorem and the intermediate value theorem. Since $F$ is a degree one map, and $\sup_{y\in [x-1, x]} F(y)-\inf_{y\in [x-1, x]} F(y)\geq 1$, we then have
$(G \circ F)_U(x)=\sup\{G(z): z\leq F(x^*)\}$.
Hence $(G\circ F)_U=G_U \circ F_U$. A similar argument works for lower maps.
\end{proof}

\subsection{Random interpolating maps}

For a given map $G\in\mathrm{End}_1(\mathbb{R})$, the \emph{plateau set} of $G$ is the open set given by 
\begin{align}\label{eq:plateau}
\mathrm{Pl}(G) \coloneq \left\{x\in\mathbb{R}: \exists\, \varepsilon>0,\ G([x-\varepsilon,x+\varepsilon])=\{G(x)\}\,\right\}.
\end{align}
In \cite{Glendinningetal2009}, Glendinning et al.~constructed, for a given map $G\in\mathrm{End}_1(\mathbb{R})$, a family of monotone maps $\{G_a\}_{a\in[0,1]} \subset \mathrm{Mon}_1(\mathbb{R})$ interpolating between $G_0 = G_L$ and $G_1 = G_U$, defined by $G_a(x) \coloneq \inf_{y\geq x} \sup_{z\in [y-a, y]} G(z)$. This deterministic homotopy is pointwise continuous in $a$ and satisfies the \emph{plateau property}: if $G_a(x) \neq G(x)$, then $G_a$ is locally constant around $x$. We now transfer this construction to the random setting.

\begin{definition}\label{def:interp}
For each $F\in \widetilde{\mathcal{E}}(\Omega)$ and $a\in[0,1]$, the \emph{random $a$-interpolating map} $F_a$ is defined by
\begin{equation}\label{eq:interpolatingmap}
F_{\omega, a}(x)\coloneq\inf_{y\geq x} \sup_{z\in [y-a, y]} F_\omega(z),
\end{equation}
for each $\omega\in\Omega$ and $x\in\mathbb{R}$. We refer to the set $\{F_a\}_{a\in[0,1]}$ as the \emph{random interpolating map family} of $F$. 
\end{definition}

Note that  for each $\omega\in\Omega$ and $x\in\mathbb{R}$,
\begin{align*}
F_{\omega,0}(x)& = \inf_{y\geq x} \sup_{z\in[y,y]}F_\omega(z)= \inf_{y\geq x}F_\omega(y) = (F_\omega)_L(x) \\    
F_{\omega,1}(x)&= \inf_{y\geq x} \sup_{z\in [y-1,y]} F_\omega(z)=\sup_{z\in[x-1,x]}F_\omega(z)=\sup_{z\leq x}F_\omega(z)=(F_\omega)_U(x),
\end{align*}
since $F_\omega$ satisfies the degree one property. Thus we call the map $F_L\coloneq F_0$ the \emph{random lower map} $F_L$ and $F_U\coloneq F_1$ the \emph{random upper map}. Given $F\in \widetilde{\mathcal{E}}(\Omega)$, by Lemma \ref{lem:upperlowercomposition} we have that the lower or upper map of the composition $F^{(n)}_\omega$ is equal to the composition of the individual lower or upper maps: that is, for each $\omega\in\Omega$ and $n\geq 1$ we have
\begin{align}\label{eq:loweruppercocycle}
(F^{(n)}_{\omega})_L=(F_{L})_{\omega}^{(n)}
\quad\textrm{and}\quad
(F^{(n)}_{\omega})_U=(F_{U})_{\omega}^{(n)}.
\end{align}
We can therefore drop the parentheses and simply write $F^{(n)}_{\omega, L}$ and $F^{(n)}_{\omega, U}$ respectively for the lower and upper maps of $F^{(n)}_\omega$. For $a\in (0, 1)$, however, the interpolating construction does not in general commute with composition. Hence the notation $F^{(n)}_{\omega, a}$ will always mean the random composition of the one-step interpolating maps along the base orbit:
\begin{align}
F^{(n)}_{\omega,a}\coloneq F_{\sigma^{n-1}\omega,a} \circ \cdots \circ F_{\sigma\omega,a} \circ F_{\omega,a},
\end{align}
for $a\in [0,1]$ and $\omega\in\Omega$.

The random interpolating map family has the following properties, mirroring those demonstrated in \cite{Glendinningetal2009} in the deterministic case.

\begin{proposition}\label{pro:interpolatingmap}
Let $F \in \widetilde{\mathcal{E}}(\Omega)$. The random interpolating map family $\{F_a\}_{a\in [0, 1]}$ satisfies the following properties: 
\begin{enumerate}
\item (Measurability and integrability) For each $a\in [0, 1]$, $F_a \in \widetilde{\mathcal{M}}(\Omega)$.
\item ($L^1$-continuity) The map $a \mapsto F_a$ is continuous with respect to the uniform metric $D(\cdot, \cdot)$ on $\widetilde{\mathcal{M}}(\Omega)$.
\item (Monotonicity in the interpolating parameter) For each $\omega\in\Omega$, the map $a\mapsto F_{\omega, a}$ is non-decreasing.
\item (Fibrewise plateau property) For each $\omega\in \Omega$ and $a\in [0, 1]$, if $F_{\omega, a}(x)\neq F_{\omega}(x)$, then $x \in \mathrm{Pl}(F_{\omega, a})$. 
\end{enumerate}
\end{proposition}

\begin{proof}
For item 1, to establish that $F_a \in \widetilde{\mathcal{M}}(\Omega)$ we must verify that $F_{\omega, a}$ is both measurable with respect to $\omega$ and integrable. A priori, the definition of $F_{\omega,a}(x)$ involves extrema over continuous intervals, which poses potential measurability issues. However, since the fibre map $F_\omega$ is continuous, the extrema over any interval coincide with the extrema over a countable dense subset. Furthermore, since $F_\omega$ has degree one, the outer infimum may be naturally restricted to $y \in [x, x+1]$. We can thus explicitly rewrite the interpolation as
\begin{align}
F_{\omega,a}(x) = \inf_{y \in [x,x+1] \cap \mathbb{Q}} \sup_{z \in [y-a,y] \cap \mathbb{Q}} F_\omega(z). 
\end{align}
Because pointwise supremum and infimum over countable index sets preserve measurability, the dependence $\omega \mapsto F_{\omega,a}(x)$ is measurable. 

For integrability, observe that by construction, the interpolating map is bounded by the lower and upper monotone envelopes: for each $\omega\in\Omega$ and $a\in [0,1]$, we have $F_{\omega, L} \leq F_{\omega, a}\leq F_{\omega, U}$. Consequently, its maximal displacement is uniformly bounded by the maximal displacement of the fibre map: that is, $\| \Delta F_{\omega, a} \| \le \| \Delta F_\omega \|$. Since $F \in \widetilde{\mathcal{E}}(\Omega)$, the norm $\| \Delta F_\omega \|$ belongs to $L^1(\Omega)$, which immediately guarantees that $\int_\Omega \| \Delta F_{\omega, a} \|\,\mathrm{d}\mathbb{P} < \infty$. Thus, $F_a \in \widetilde{\mathcal{M}}(\Omega)$.

For item 2, fix $a\in[0,1]$ and let $a_n\to a$. For each fixed $\omega\in\Omega$, the fibre map $F_\omega$ is continuous and has degree one. Hence, on every interval of the form $[y-1,y]$, it is uniformly continuous. Therefore, for each fixed $\omega$, we have $\|F_{\omega,a_n}-F_{\omega,a}\|\to 0$ as $n\to\infty$.
The interpolating maps $F_{\omega,a}(x)$ satisfy
\begin{align*}
F_{\omega,a}(x) = \inf_{y\in[x,x+1]}\sup_{z\in[y-a,y]}F_\omega(z),
\end{align*}
and similarly for $F_{\omega,a_n}$. Hence
\begin{align*}
\|F_{\omega,a_n}-F_{\omega,a}\| \le \sup_{x\in\mathbb R}\sup_{y\in[x,x+1]} \sup\bigl\{ |F_\omega(z)-F_\omega(z')|: z\in[y-a_n,y],\ z'\in[y-a,y] \bigr\},
\end{align*}
which tends to $0$ as $n\to\infty$.

Moreover, by the triangle inequality and the bound $\|\Delta F_{\omega,a}\| \le \|\Delta F_\omega\|$, we have
\begin{align*}
\|F_{\omega,a_n}-F_{\omega,a}\| \le \|F_{\omega,a_n}-\mathrm{Id}\| + \|F_{\omega,a}-\mathrm{Id}\| \le 2\|\Delta F_\omega\|.
\end{align*}
Since $\|\Delta F\|\in L^1(\Omega)$, the dominated convergence theorem yields
\begin{align*}
D(F_{a_n},F_a) = \int_\Omega \|F_{\omega,a_n}-F_{\omega,a}\|\,\mathrm{d}\mathbb P(\omega) \to 0.
\end{align*}
Thus $a\mapsto F_a$ is continuous in $\widetilde{\mathcal{M}}(\Omega)$.

For item 3, fix $\omega\in \Omega$ and $x\in \mathbb{R}$. If $0\leq a\leq b\leq 1$, then for every $y\in [x, x+1]$, we have $[y-a, y]\subset [y-b, y]$. Hence 
\begin{align}
\sup_{z\in [y-a, y]} F_\omega(z)\leq \sup_{z\in [y-b, y]} F_\omega(z).
\end{align}
Taking the infimum over $y\in [x, x+1]$, we obtain $F_{\omega, a}(x)\leq F_{\omega, b}(x)$.

Finally, for item 4, we fix $\omega\in \Omega$ and $a\in [0,1]$ and consider the map $F_{\omega, a}\in\mathrm{Mon}_1(\mathbb{R})$. Following the deterministic argument of \cite{Glendinningetal2009} applied to the single map $F_\omega\in\mathrm{end}_1(\mathbb{R})$, we have $x\in \mathrm{Pl}(F_{\omega,a})$ if $F_{\omega,a}(x)\neq F_\omega(x)$.
\end{proof}

\subsection{Coincidence sets for random interpolating maps}
To establish the compactness of the random rotation set, we analyse the orbits of random interpolating maps $F_{\omega, a}$ that coincide with orbits of the original map $F_\omega$. For this, we introduce the following sets.

\begin{definition}\label{def:coincidencesets}
Let $F \in \widetilde{\mathcal{E}}(\Omega)$ and fix $a \in [0, 1]$. For each $\omega \in \Omega$ and $n\in\mathbb{N}_0$, we define the \emph{$n$-step coincidence set} $S^n_{\omega,a}$ iteratively as follows:
\begin{align*}
S^0_{\omega,a} &\coloneq \{x \in \mathbb{R} : F^{(0)}_{\omega,a}(x) = F^{(0)}_\omega (x)\} = \mathbb{R}, \\
S^1_{\omega,a} &\coloneq \{x \in \mathbb{R} : F_{\omega,a}(x) = F_\omega(x)\}, \\
    &\ \ \vdots \\
S^n_{\omega,a} &\coloneq \left\{x \in \mathbb{R} : F^{(i)}_{\omega,a}(x) = F^{(i)}_\omega(x), \ \forall i = 1, \dots, n\right\}.
\end{align*}
Furthermore, the \emph{global coincidence set} is defined as $S_{\omega,a} \coloneq \bigcap_{n=1}^\infty S^n_{\omega,a}$.
\end{definition}

Coincidence sets have the following properties.

\begin{proposition}\label{pro:coincidence}
For each $n \in \mathbb{N}_0$, $l\in\mathbb{Z}$, $\omega \in \Omega$, and $a \in [0, 1]$, the $n$-step coincidence set $S^n_{\omega,a}$ satisfies:
\begin{enumerate}
\item (Nestedness) $S^{n+1}_{\omega,a} \subset S^n_{\omega,a}$;
\item (Recursive structure) $S^{n+1}_{\omega,a} = S^n_{\omega,a} \cap \left(F^{(n)}_{\omega,a}\right)^{-1}\left(S^1_{\sigma^n\omega,a}\right)$;
\item (Full coverage) $F^{(n)}_{\omega,a}(S^n_{\omega,a})=\mathbb{R}$ and so $S^n_{\omega,a}\neq \emptyset$;
\item (Integer translational invariance) $S^n_{\omega,a}+l=S^n_{\omega,a}$;
\item (Closedness) $S^n_{\omega,a}$ is closed;
\item (Global coincidence) $S_{\omega,a}$ is a non-empty closed set invariant under integer translation.  
\end{enumerate}
\end{proposition}
\begin{proof}
For item 1, note that $S^{n+1}_{\omega, a}$ satisfies
\begin{align*}
S^{n+1}_{\omega, a}=S^n_{\omega, a} \cap \left\{x \in \mathbb{R}: F^{(n+1)}_{\omega, a}(x)=F^{(n+1)}_\omega(x)\right\}.
\end{align*}
So $S^{n+1}_{\omega, a}$ is a subset of $S^n_{\omega, a}$.

We prove item 2 by induction. For the initial step, note that $S^0_{\omega, a}\cap (F^{(0)}_{\omega, a})^{-1}(S^1_{\omega, a})=S^0_{\omega, a}\cap S^1_{\omega, a}=S^1_{\omega, a}$ since $S^1_{\omega, a}\subset S^0_{\omega, a}$ by item 1. For the induction step, we assume that $S^k_{\omega, a}=S^{k-1}_{\omega, a}\cap (F_{\omega, a}^{(k-1)})^{-1}(S^1_{\sigma^{k-1}\omega, a})$ is true and consider $S^{k+1}_{\omega, a}$. We have
\begin{align*}
S^{k+1}_{\omega, a}=S^{k}_{\omega, a} \cap \{x \in \mathbb{R}: F^{(k+1)}_{\omega,a}(x)=F^{(k+1)}_\omega(x)\}.
\end{align*}
Suppose $x \in S^k_{\omega, a} \cap (F^{(k)}_{\omega, a})^{-1}(S^1_{\sigma^k\omega, a})$. Since $x\in S^k_{\omega, a}$, we have $F^{(i)}_{\omega, a}(x)=F^{(i)}_{\omega}(x)$ for each $i\in\{0, \ldots, k\}$. As $x\in  (F^{(k)}_{\omega, a})^{-1}(S^1_{\sigma^k\omega, a})$, we know that $F^{(k)}_{\omega, a}(x)\in S^1_{\sigma^k\omega, a}$. Let $y$ denote the point $F^{(k)}_{\omega, a}(x)=F^{(k)}_{\omega}(x)$. Since $y\in S^1_{\sigma^k\omega, a}$, we have
\begin{align*}
F^{(k+1)}_{\omega, a}(x)=F_{\sigma^k\omega, a}(F^{(k)}_{\omega, a}(x))=F_{\sigma^k\omega,a}(y)=F_{\sigma^k\omega}(y)=F_{\sigma^k\omega}(F^{(k)}_{\omega}(x))=F^{(k+1)}_{\omega}(x).
\end{align*}
Hence $F^{(k+1)}_{\omega, a}(x)=F^{k+1}_\omega(x)$ and so $x \in S^{k+1}_{\omega, a}$.

Conversely, suppose that $x\in S^{k+1}_{\omega, a}$. We need to show $x\in S^k_{\omega, a}\cap (F^{(k)}_{\omega, a})^{-1}(S^1_{\sigma^k\omega, a})$. We have
\begin{align*}
x\in S^{k+1}_{\omega, a}=\left\{x\in \mathbb{R}: F^{(i)}_{\omega, a}(x)=F^{(i)}_{\omega}(x), \ \forall i=1,\ldots, k+1\right\}.
\end{align*}
Since $F^{(i)}_{\omega, a}(x)=F^{(i)}_{\omega}(x)$ for all $i=1, \ldots, k$, we have $x \in S^k_{\omega, a}$. Let $y$ denote the point $F^{(k)}_{\omega, a}(x)=F^{(k)}_{\omega}(x)$. Since $x \in S^{k+1}_{\omega, a}$, we have
$F_{\sigma^k \omega, a}(y)=F_{\sigma^k\omega}(y)$. Hence $y \in S^1_{\sigma^k\omega, a}$. That is, $F^{(k)}_{\omega, a}(x)=F^{(k)}_{\omega}(x)\in S^1_{\sigma^k\omega, a}$. This implies $x \in S^k_{\omega, a} \cap (F^{(k)}_{\omega, a})^{-1}(S^1_{\sigma^k\omega, a})$, which completes item 2.

For item 3, we show that $F^{(n)}_{\omega,a}(S^n_{\omega,a})=\mathbb{R}$ by induction on $n\geq 1$. 
For the initial step, take $c\in\mathbb{R}$. We have $\lim_{x\to\pm\infty}F_{\omega,a}(x)=\pm\infty$, since $F_{\omega,a}$ has the degree one property. So there exist $k_1,k_2\in\mathbb{Z}$, with $k_1<k_2$, such that $F_{\omega,a}(k_1)<c<F_{\omega,a}(k_2)$. Since $F_{\omega,a}$ is continuous, by the intermediate value theorem there exists $y\in [k_1,k_2]$ such that $F_{\omega,a}(y)=c$. If $y\in S^1_{\omega,a}$, then $c\in F_{\omega, a}(S^1_{\omega, a})$ and we are done, so assume $y\notin S^1_{\omega,a}$. Thus $F_{\omega,a}(y)\neq F_\omega(y)$, so $y\in\mathrm{Pl}(F_{\omega,a})$ by Proposition~\ref{pro:interpolatingmap}(4). Take $z$ to be the left endpoint of the plateau of $F_{\omega,a}$ that contains $y$. By \eqref{eq:plateau},  the plateau set $\mathrm{Pl}(F_{\omega,a})$ is open, so $z\in S^1_{\omega,a}$, and thus $F_{\omega,a}(z)=F_{\omega}(y)$. So we have $c=F_{\omega,a}(z)\in F_{\omega, a}(S^1_{\omega, a})$. Since $c\in\mathbb{R}$ was arbitrary, we have $F_{\omega,a}(S^1_{\omega,a})=\mathbb{R}$.
For the induction step, assume that $F^{(i)}_{\omega,a}(S^i_{\omega,a})=\mathbb{R}$ for all $i=1,\ldots,n$. Then by item 2,
\begin{align*}
F^{(n+1)}_{\omega,a}(S^{n+1}_{\omega,a}) & = F_{\sigma^n\omega}(F^{(n)}_{\omega,a}(S^n_{\omega,a}\cap (F^{(n)}_{\omega,a})^{-1}(S^1_{\sigma^n\omega,a}))) \\
&= F_{\sigma^n\omega}(F^{(n)}_{\omega,a}(S^n_{\omega,a})\cap S^1_{\sigma^n\omega,a}) \\
&= F_{\sigma^n\omega}(S^1_{\sigma^n\omega,a}) \\
&= \mathbb{R}
\end{align*}
as required, using the induction hypothesis with $i=n$ and then with $i=1$. Moreover, $S^0_{\omega,a}=\mathbb{R}$ and $F^{(0)}_{\omega,a}=\mathrm{Id}$, so the result is trivially true when $n=0$. Thus $S^n_{\omega,a}$ is non-empty for all $n\in\mathbb{N}_0$.

For item 4, we have $S^0_{\omega, a}+l=\mathbb{R}=S^0_{\omega, a}$ for each $l\in\mathbb{Z}$. For each $a \in [0, 1]$,  $\omega \in \Omega$ and $n \in \mathbb{N}$, we define a function $G_{n, \omega, a}:\mathbb{R}\rightarrow \mathbb{R}$ by
$G_{n, \omega, a} \coloneq F^{(n)}_{\omega, a}-F^{(n)}_{\omega}$. So $G_{n, \omega, a}$ is continuous, and for each $l\in \mathbb{Z}$,
\begin{align*}
G_{n, \omega, a}(x+l)&=F^{(n)}_{\omega, a}(x+l)-F^{(n)}_\omega(x+l)\\
    &=F^{(n)}_{\omega, a}(x)+l-(F^{(n)}_\omega(x)+l)\\
    &=G_{n, \omega, a}(x).
\end{align*}
Hence $G_{n, \omega, a}$ is a $1$-periodic function, and so $G^{-1}_{n, \omega, a}(\{0\})+l=G^{-1}_{n, \omega, a}(\{0\})$ for each $l\in\mathbb{Z}$. For each $n\in \mathbb{N}$, the set $S^n_{\omega, a}$ can be expressed as an intersection of integer-translation invariant sets:
\begin{align}\label{eq:intersectG}
S^n_{\omega, a}=\bigcap_{i=1}^n G^{-1}_{i, \omega, a}(\{0\}).
\end{align}
Hence $S^n_{\omega, a}+l=S^n_{\omega, a}$ for each $l\in\mathbb{Z}$.

For item 5, we have $S^0_{\omega, a}=\mathbb{R}$, which is closed. Since $G_{i, \omega, a}$ is continuous for each $i\in \mathbb{N}$, the pre-image sets $G^{-1}_{i, \omega, a}(\{0\})$ are closed. Thus $S^n_{\omega, a}$ is an intersection of closed sets by \eqref{eq:intersectG}, and so is also closed.

For item 6, it follows from the previous items that for each $l\in\mathbb{Z}$ the sets $S^n_{\omega, a}\cap [l, l+1]$ are non-empty compact and nested sets. Therefore, by the Cantor intersection theorem, their intersection $S_{\omega, a}\cap [l, l+1]$ is non-empty. 
Since the sets $S^n_{\omega, a}$ are integer translation invariant for each $n\in\mathbb{N}$, their intersection $S_{\omega, a}$ is also integer translation invariant.
\end{proof}

As the global coincidence set $S_{a,\omega}$ is non-empty for each $a\in [0, 1]$, there exists at least one point for which the orbit with respect to $F$ coincides with the orbit with respect to $F_a$.

\subsection{Random rotation sets}
We now prove the interval structure of the random rotation set. In the deterministic and uniquely ergodically-forced settings, related interval descriptions are obtained by global topological arguments such as \cite{Glendinningetal2009}. In the present measurable random setting, however, we cannot rely on compactness and continuity of the base system, continuity of the fibre map dependence or on the existence of invariant minimal sets. Instead, we work directly with the random interpolating family and over an ergodic measure-preserving dynamical system.

We begin this section by proving that there is a common full measure set on which all maps in the interpolating family achieve their mean random rotation numbers.

\begin{lemma}(Simultaneous convergence for the interpolating family)\label{lem:simconvergence}
Let $(\Omega, \mathcal{F}, \mathbb{P}, \sigma)$ be an ergodic probability measure-preserving dynamical system. Let $F\in \widetilde{\mathcal{E}}(\Omega)$ be a lift of a random circle endomorphism over $(\Omega, \mathcal{F}, \mathbb{P}, \sigma)$ and let  $\{F_a\}_{a\in[0,1]}$ be the random interpolating family of $F$. The mapping $a\mapsto \rho(F_a)$ is continuous and non-decreasing.
Moreover, there exists a $\sigma$-invariant full measure set $\Omega^*\subset \Omega$ such that for every $\omega\in \Omega^*$, every $a\in [0, 1]$ and every $x\in \mathbb{R}$,
\begin{align*}
\lim_{n\to \infty}\frac{F^{(n)}_{\omega, a}(x)-x}{n}=\rho(F_a).
\end{align*}
\end{lemma}
\begin{proof}

We first note that the dependence $a\mapsto F_a$ of the interpolating maps on the parameter is a continuous map from $[0, 1] \to \widetilde{\mathcal{M}}(\Omega)$ by Proposition~\ref{pro:interpolatingmap}(2). As proven by Li and Lu, for lifts of random monotone circle maps, the random rotation number $\rho$ is a continuous map $\widetilde{\mathcal{M}}(\Omega)\to L^1(\Omega)$. 
Hence
\begin{align}
a\mapsto \int_\Omega \rho_{F_a}\,\mathrm{d}\mathbb{P}=\rho(F_a), 
\end{align}
the dependence of the mean random rotation number of $F_a$ on the parameter $a$, is a continuous map from the closed interval $[0,1]$ to $\mathbb{R}$.
By Proposition~\ref{pro:interpolatingmap}(3), $a\mapsto F_{\omega,a}$ is non-decreasing for each $\omega\in\Omega$. Since monotonicity is preserved by composition, $a\mapsto F^{(n)}_{\omega,a}$ is non-decreasing for each $\omega\in\Omega$ and $n\in\mathbb{N}$. So given $a,a'\in [0,1]$ with $a<a'$, we have 
\begin{align*}
\int_\Omega \frac{F^{(n)}_{\omega,a}(0)}{n}\,\mathrm{d}\mathbb{P}(\omega) \leq \int_\Omega \frac{F^{(n)}_{\omega,a'}(0)}{n}\,\mathrm{d}\mathbb{P}(\omega)
\end{align*}
for each $n\in\mathbb{N}$. By \eqref{eq:LiLu}, we have almost sure convergence of the ratios within each integral as $n\to\infty$, yielding $\rho(F_a)\leq \rho(F_{a'})$. Hence $a\mapsto \rho(F_a)$ is non-decreasing.

Let $\mathcal Q=\mathbb Q\cap [0, 1]$. For each $q\in \mathcal Q$, we have $F_q\in\widetilde{\mathcal{M}}(\Omega)$ by Proposition~\ref{pro:interpolatingmap}(1), and so by \eqref{eq:LiLu} there exists a $\sigma$-invariant full-measure set $\Omega_q$ such that, for every $\omega\in \Omega_q$ and every $x\in \mathbb{R}$,
\begin{align*}
\lim_{n\to\infty}
\frac{F_{\omega,q}^{(n)}(x)-x}{n}=\rho(F_q).
\end{align*}
Let $\Omega^*=\bigcap_{q\in\mathcal Q}\Omega_q$. Then $\Omega^*$ is $\sigma$-invariant and has full measure.

Fix $a\in [0,1]$. If $a\in \mathcal{Q}$, then the conclusion is immediate. Otherwise, choose $p, q\in \mathcal{Q}$ with $p<a<q$. By Proposition~\ref{pro:interpolatingmap}(3), the interpolating family is monotonically increasing in the parameter, so for every $\omega\in \Omega$, we have $F_{\omega,p}\leq F_{\omega,a}\leq F_{\omega,q}$. Since monotonicity of the interpolating maps is preserved by composition, we have $F_{\omega, p}^{(n)} \leq F_{\omega, a}^{(n)} \leq F_{\omega, q}^{(n)}$ for all $n\in\mathbb{N}$.
Hence, for each $\omega\in\Omega^*$ and $x\in\mathbb{R}$,
$$
\rho(F_p)
\leq
\liminf_{n\to\infty}
\frac{F_{\omega,a}^{(n)}(x)-x}{n}
\leq
\limsup_{n\to\infty}
\frac{F_{\omega,a}^{(n)}(x)-x}{n}
\leq
\rho(F_q).
$$
Letting $p\uparrow a$ and $q\downarrow a$ through rational parameters, we obtain the desired limit since $a\mapsto \rho(F_a)$ is continuous. The endpoint cases $a=0,1$ follow directly since $0,1\in\mathcal Q$.
\end{proof}

The following proposition constitutes the proof of part 2 of Theorem \ref{thm:rotset}, the existence of a full measure set on which all values in the random rotation set are achieved as asymptotic average displacements along actual orbits of an integrable lift of a random circle endomorphism.

\begin{proposition}(Simultaneous realisation of random rotation values)\label{pro:simrealisation}
For every $\omega\in\Omega^*$ and every $c\in[\rho(F_L),\rho(F_U)]$, there exists $x_c(\omega)\in[0,1]$ satisfying
$$
\lim_{n\to\infty} \frac{F_\omega^{(n)}(x_c(\omega))-x_c(\omega)}{n}=c.
$$
\end{proposition}
\begin{proof}
By Lemma~\ref{lem:simconvergence}, $a\mapsto \rho(F_a)$ is a continuous and non-decreasing map from the closed interval $[0,1]$ to $\mathbb{R}$. Thus $\{\rho(F_a):a\in [0, 1]\}=[\rho(F_L),\rho(F_U)]$, since the lower random map is $F_L=F_0$ and the upper random map is $F_U=F_1$.
Moreover, given $c\in [\rho(F_L),\rho(F_U)]$, by the intermediate value theorem, there exists $\alpha \in [0, 1]$ such that $\rho(F_{\alpha})=c$. 

By Proposition~\ref{pro:coincidence}, the global coincidence set $S_{\omega,\alpha}\coloneq \{x\in\mathbb R: F_\omega^{(n)}(x)=F_{\omega,\alpha}^{(n)}(x)\text{ for all }n\geq 1\}$ satisfies $S_{\omega, \alpha}\cap [0, 1]\neq \emptyset$. Choose $x_c(\omega)\in S_{\omega, \alpha} \cap [0, 1]$. Then
$F^{(n)}_{\omega}(x_c(\omega))=F^{(n)}_{\omega, \alpha}(x_c(\omega))$ for all \mbox{$n\geq 1$}. Therefore,
\begin{align*}
\lim_{n\to \infty}\frac{F^{(n)}_\omega(x_c(\omega))-x_c(\omega)}{n}=\lim_{n\to \infty}\frac{F^{(n)}_{\omega, \alpha}(x_c(\omega))-x_c(\omega)}{n}=\rho(F_\alpha)=c
\end{align*}  
for all $\omega\in\Omega^*$ by Lemma~\ref{lem:simconvergence}.
\end{proof}

\begin{remark}
The point $x_c(\omega)$ is chosen separately for each pair $(\omega, c)$. Note that we do not require the map $\omega\mapsto x_c(\omega)$ to be measurable or jointly measurable in $(\omega, c)$, since we only use the existence of such a point for each fixed $\omega\in\Omega^*$ in order to realise $c$ by an orbit of the random circle endomorphism. 
\end{remark}

Recall the definition of random rotation set is defined by taking closure of the upper limits of the average displacements along orbits. We now complete the proof of Theorem~\ref{thm:rotset}.

\begin{proof}[Proof of Theorem~\ref{thm:rotset}(1)]
For convenience, we use the notation
$$
A_n(\omega,x)\coloneq \frac{F_\omega^{(n)}(x)-x}{n}.
$$
Recall that by Lemma~\ref{lem:upperlowercomposition} iterated along the base orbit, we have \eqref{eq:loweruppercocycle}, which states that for each $n\in\mathbb{N}$ and $\omega\in\Omega$ we have $(F^{(n)}_{\omega})_L=F^{(n)}_{\omega, L}$ and $(F^{(n)}_{\omega})_U=F^{(n)}_{\omega, U}$.
Thus $F^{(n)}_{\omega, L}$ and $F^{(n)}_{\omega, L}$ are respectively the lower and upper maps of $F^{(n)}_{\omega}$, and so $F^{(n)}_{\omega, L}\leq F^{(n)}_\omega\leq F^{(n)}_{\omega, U}$ for every $n\geq 1$. For each $\omega\in\Omega^*$, the $\sigma$-invariant full-measure set given by Lemma~\ref{lem:simconvergence}, we have 
$$
\lim_{n\to \infty}\frac{F^{(n)}_{\omega, L}(x)-x}{n}=\rho(F_L) \quad\textrm{and}\quad \lim_{n\to \infty}\frac{F^{(n)}_{\omega, U}(x)-x}{n}=\rho(F_U),
$$
for each $x\in\mathbb{R}$. Hence
$$
\rho(F_L)=\lim_{n\to\infty}\frac{F^{(n)}_{\omega, L}(x)-x}{n}\leq \liminf_{n\to \infty}A_n(\omega, x)\leq \limsup_{n\to \infty}A_n(\omega, x)\leq \lim_{n\to\infty}\frac{F^{(n)}_{\omega, U}(x)-x}{n}=\rho(F_U),
$$
which implies that 
$\mathrm{Rot}^+_F(\omega)\subset [\rho(F_L), \rho(F_U)]$ and $\mathrm{Rot}^-_F(\omega)\subset [\rho(F_L), \rho(F_U)]$.
Moreover, if $c\in \mathrm{Rot}^{\mathrm{ord}}_F(\omega)$, then $c$ is the limit of $A_n(\omega, x)$ for some $x$ and the same inequality gives $c\in [\rho(F_L), \rho(F_U)]$. Thus,
\begin{equation}\label{eq:rotsetcontainment}
\mathrm{Rot}^+_F(\omega)\cup\mathrm{Rot}^-_F(\omega)\cup\mathrm{Rot}^{\mathrm{ord}}_F(\omega)\subset [\rho(F_L), \rho(F_U)].
\end{equation}
For the reverse inclusion, let $c\in [\rho(F_L), \rho(F_U)]$ be given. 
Since $\omega\in\Omega^*$, by Proposition~\ref{pro:simrealisation} there exists $x_c(\omega)\in [0, 1]$ such that $ \lim_{n\to \infty} A_n(\omega, x_c(\omega))=c$.
Hence $c\in \mathrm{Rot}^{\mathrm{ord}}_F(\omega)$. Since the same orbit has both upper and lower limits equal to $c$, we also have
$c\in \mathrm{Rot}^+_F(\omega)$ and $c\in \mathrm{Rot}^-_F(\omega)$.
Since $c\in [\rho(F_L), \rho(F_U)]$ was arbitrary,
\begin{equation}\label{eq:reverserrt}
[\rho(F_L), \rho(F_U)]\subset\mathrm{Rot}^+_F(\omega)\cap \mathrm{Rot}^-_F(\omega)\cap \mathrm{Rot}^{\mathrm{ord}}_F(\omega).
\end{equation}

It remains to show the equality of the subsequential definition. Let $c\in \mathrm{Rot}_F^{\mathrm{sub}}(\omega)$. 
Then there exist sequences $(n_k)\to \infty$ and $(y_k)\subset \mathbb{R}$ such that $A_{n_k}(\omega, y_k)\to c$.
Without loss of generality, we can assume $y_k\in[0,1)$ for all $k$, since $F^{(n_k)}(y_k)-y_k=F^{(n_k)}(\mathrm{frac}(y_k))-\mathrm{frac}(y_k)$ by the degree one property, where $\mathrm{frac}(\cdot)$ denotes the fractional part.
For each $k\in\mathbb{N}$, the ordering of the lower and upper maps gives
\begin{align}\label{eq:subsequence}
\frac{F_{\omega,L}^{(n_k)}(y_k)-y_k}{n_k}\leq A_{n_k}(\omega,y_k) \leq \frac{F_{\omega,U}^{(n_k)}(y_k)-y_k}{n_k}.
\end{align}
As $0\leq y_k<1$ for each $k\in\mathbb{N}$ we have 
$$
\frac{F^{(n_k)}_{\omega,L}(0)}{n_k}\leq \frac{F^{(n_k)}_{\omega,L}(y_k)}{n_k} \leq \frac{F^{(n_k)}_{\omega,L}(1)}{n_k} = \frac{F^{(n_k)}_{\omega,L}(0)}{n_k} +\frac{1}{n_k},
$$
where the outer terms tend to $\rho(F_L)$ as $k\to\infty$ for all $\omega\in\Omega^*$. Hence the left-hand term of \eqref{eq:subsequence} tends to $\rho(F_L)$ as $k\to\infty$ for all $\omega\in\Omega^*$. By a similar argument using the upper map $F_U$, the right-hand term of \eqref{eq:subsequence} tends to $\rho(F_U)$ as $k\to\infty$ for all $\omega\in\Omega^*$. Hence $c\in [\rho(F_L), \rho(F_U)]$. Thus
$\mathrm{Rot}_F^{\mathrm{sub}}(\omega)\subset [\rho(F_L), \rho(F_U)]$.
Conversely, if $c\in [\rho(F_L), \rho(F_U)]$, then by Proposition~\ref{pro:simrealisation} there exists $x_c(\omega)$ such that $A_n(\omega, x_c(\omega))\to c$. Taking $y_k=x_c(\omega)$ and $n_k=k$, we obtain $c\in \mathrm{Rot}_F^{\mathrm{sub}}(\omega)$. Therefore
\begin{equation}\label{eq:subrrt}
\mathrm{Rot}_F^{\mathrm{sub}}(\omega)=[\rho(F_L), \rho(F_U)].
\end{equation}

Recall that the random rotation set $\mathrm{Rot}_F(\omega)$ is defined as the closure of $\mathrm{Rot}^+_F(\omega)$. Combining \eqref{eq:subrrt} with \eqref{eq:rotsetcontainment} and \eqref{eq:reverserrt}, we conclude that for all $\omega\in \Omega^*$,
$$
\mathrm{Rot}_F(\omega)=\mathrm{Rot}^+_F(\omega)=\mathrm{Rot}^-_F(\omega)=\mathrm{Rot}^{\mathrm{ord}}_F(\omega)=\mathrm{Rot}^{\mathrm{sub}}_F(\omega)=[\rho(F_L), \rho(F_U)].
$$
\end{proof}

\subsection{Realisation of subintervals}\label{sec:realsubint}
The realisation statement in Theorem~\ref{thm:rotset}(2) produces an orbit with prescribed asymptotic average displacement for each single value in the rotation interval on a common full-measure set. In this section, we prove another realisation result, Theorem \ref{thm:realupperlower} which provides a single orbit, for any $a,b\in\mathrm{Rot}(F)$ with $a\leq b$, for which the lower and upper asymptotic displacement averages are precisely $a$ and $b$ respectively. The full-measure set in this result may depend on the prescribed pair $(a, b)$.

To prove this, we use a version of the coincidence construction based on multiple interpolating parameters. This allows us to demonstrate a point for which the orbit with respect to $F$ alternates between shadowing $F_\alpha$ and $F_\beta$ on consecutive blocks of time, where $\rho(F_\alpha)=a$ and $\rho(F_\beta)=b$.

\begin{definition}(Block coincidence sets) \label{def:blockcoincidence}
Let $F \in \widetilde{\mathcal{E}}(\Omega)$ and fix $\alpha, \beta \in [0, 1]$. Let $(\gamma_m)_{m\geq 1}$ be any sequence with $\gamma_m\in \{\alpha, \beta\}$, and let $(n_m)_{m\geq 1}$ be a sequence of positive integers. For any $m\geq 1$ and $\omega\in \Omega$, the \emph{block coincidence set} $C^{(m)}_\omega$ is defined recursively by $C^{(0)}_\omega\coloneq\mathbb{R}$ and
\begin{align}\label{eq:blockcoincidence}
C^{(m)}_\omega\coloneq C^{(m-1)}_\omega \cap (F^{(T_{m-1})}_\omega)^{-1}\left(S^{n_m}_{\sigma^{T_{m-1}}\omega, \gamma_m}\right), 
\end{align}
where $T_0\coloneq 0$ and $T_m\coloneq n_1+\cdots +n_m$. The \emph{global block coincidence set} is $C_{\omega} \coloneq \bigcap_{m=1}^\infty C^{(m)}_{\omega}$.
\end{definition}

Block coincidence sets have the following properties.

\begin{proposition}(Properties of block coincidence sets) \label{pro:blockcoincidence}
Fix $\alpha, \beta \in [0, 1]$. Let $(\gamma_m)_{m\geq 1}$ be any sequence with $\gamma_m\in \{\alpha, \beta\}$, and let $(n_m)_{m\geq 1}$ be any sequence of positive integers. 
For every $m\geq 0$, $l\in\mathbb{Z}$ and every $\omega\in \Omega$, the block coincidence set $C^{(m)}_\omega$ satisfies the following statements:
\begin{enumerate}
\item (Nestedness) $C^{(m+1)}_\omega\subset C^{(m)}_\omega$;
\item (Closedness) $C^{(m)}_\omega$ is closed;
\item (Integer translational invariance) $C^{(m)}_\omega+l=C^{(m)}_\omega$;
\item (Full coverage) $F^{(T_m)}_\omega(C^{(m)}_\omega)=\mathbb{R}$ and so $C^{(m)}_\omega\neq\emptyset$;
\item (Global block coincidence) $C_{\omega}$ is a non-empty closed set invariant under integer translation.
\end{enumerate}
\end{proposition}
\begin{proof}
Item 1 follows directly from \eqref{eq:blockcoincidence}. We prove items 2 and 3 together by induction. For $m=0$, the case is immediate. Given $m\geq 1$, for the induction hypothesis, suppose that $C^{(m-1)}_\omega$ is closed and integer translation invariant. Since $S^{n_m}_{\sigma^{T_{m-1}}\omega, \gamma_m}$ is closed and integer translation invariant by items 5 and 4 of Proposition~\ref{pro:coincidence} and $F^{(T_{m-1})}_\omega$ is continuous and degree one, the preimage set $(F^{(T_{m-1})}_\omega )^{-1}(S^{n_m}_{\sigma^{T_{m-1}}\omega,\gamma_m})$ is also closed and integer translation invariant. Hence the set $C^{(m)}_\omega$ is itself closed and integer translation invariant, since by \eqref{eq:blockcoincidence} it is the intersection of two sets with these properties. 

For item 4, we also use induction. Again, the case $m=0$ is immediate. For the induction step, suppose that $F^{(T_{m-1})}_\omega(C^{(m-1)}_\omega)=\mathbb{R}$ holds, and so $C^{(m-1)}_\omega\neq\emptyset$. Applying $F^{(T_{m-1})}_\omega$ to each side of \eqref{eq:blockcoincidence}, we have
\begin{align*}
F^{(T_{m-1})}_\omega(C^{(m)}_\omega) &= F^{(T_{m-1})}_\omega\left( C^{(m-1)}_\omega \cap (F^{(T_{m-1})}_\omega)^{-1}\left(S^{n_m}_{\sigma^{T_{m-1}}\omega, \gamma_m}\right)\right) \\
&= F^{(T_{m-1})}_\omega(C^{(m-1)}_\omega)\cap S^{n_m}_{\sigma^{T_{m-1}}\omega, \gamma_m} \\
&= S^{n_m}_{\sigma^{T_{m-1}}\omega, \gamma_m},
\end{align*}
by the induction hypothesis. Thus, by applying $F^{n_m}_{\sigma^{T_{m-1}}\omega,\gamma_m}$, we obtain
\begin{align*}
F^{(T_{m})}_\omega(C^{(m)}_\omega) =\left(F^{(n_m)}_{\sigma^{T_{m-1}}\omega,\gamma_m}\circ F^{(T_{m-1})}_\omega\right)(C^{(m)}_\omega) =F^{(n_m)}_{\sigma^{T_{m-1}}\omega,\gamma_m}\left(S^{n_m}_{\sigma^{T_{m-1}}\omega, \gamma_m}\right)=\mathbb{R}
\end{align*}
by Proposition~\ref{pro:coincidence}(3). In particular, this implies that $C^{(m)}_\omega$ is non-empty.

For item 5, it follows from the previous items that for each $l\in\mathbb{Z}$, the sets $C^{(m)}_\omega\cap [l,l+1]$, $m\geq0 $, are non-empty nested compact sets. Therefore, by the Cantor intersection theorem, their intersection $C_\omega\cap [l,l+1]$ is non-empty. Since $C_\omega$ is an intersection of integer translation invariant sets by item 3, it is also invariant under integer translation.
\end{proof}

We use block coincidence sets to prove Theorem~\ref{thm:realupperlower}.

\begin{proof}[Proof of Theorem~\ref{thm:realupperlower}]
If $a=b$, the statement reduces to Proposition \ref{pro:simrealisation}. So assume that $a<b$. By Lemma~\ref{lem:simconvergence}, $t\mapsto \rho(F_t)$ is a continuous and non-decreasing map from the closed interval $[0,1]$ to $\mathbb{R}$, so there exist $\alpha, \beta\in [0, 1]$ with $\alpha\leq \beta$ such that $\rho(F_\alpha)=a$ and $\rho(F_\beta)=b$.
Let $\Omega_{a, b}\coloneq\Omega_\alpha \cap \Omega_\beta$, where $\Omega_\alpha$ and $\Omega_\beta \subset \Omega$ are $\sigma$-invariant full-measure sets satisfying
\begin{align*}
\lim_{n\to \infty}\frac{F^{(n)}_{\omega, \alpha}(0)}{n}=a, \quad\textrm{for a.e. }\omega\in\Omega_\alpha \quad\textrm{and}\quad \lim_{n\to \infty}\frac{F^{(n)}_{\omega, \beta}(0)}{n}=b \quad \textrm{for a.e. }\omega \in \Omega_\beta.
\end{align*}

Given $\zeta\in \{\alpha, \beta\}$ and $\eta\in \Omega_\zeta$, since $F_\zeta\in\widetilde{\mathcal{M}}(\Omega)$ and $\sigma$ is ergodic, by \eqref{eq:LiLu} we have $(1/n)(F^{(n)}_{\eta,\zeta}(y)-y)\to \rho(F_\zeta)$ as $n\to\infty$ for each $y\in\mathbb{R}$. Since $\Delta F^{(n)}_{\eta,\zeta}=F^{(n)}_{\eta,\zeta}-\mathrm{Id}$ is $1$-periodic, given $\varepsilon>0$, there exists $N_\zeta(\eta,\varepsilon)\in\mathbb{N}$ such that
\begin{equation}\label{eq:pointsup}
\sup_{y \in \mathbb{R}}\left|\frac{F^{(n)}_{\eta, \zeta}(y)-y}{n}-\rho(F_\zeta)\right|<\varepsilon \quad \textrm{for every }n\geq N_\zeta(\eta, \varepsilon).
\end{equation}

Fix $\omega\in \Omega_{a, b}$. We consider the composition formed by switching back and forth between the $\alpha$- and $\beta$-interpolating maps on alternate blocks of time:
\begin{align}\label{eq:blockcomposition}
\cdots
\circ
\overbrace{F_{\sigma^{T_3-1}\omega,\alpha} \circ \cdots \circ F_{\sigma^{T_2}\omega,\alpha}}^{n_3\textrm{ times}} 
\circ
\overbrace{F_{\sigma^{T_2-1}\omega,\beta} \circ \cdots \circ F_{\sigma^{T_1}\omega,\beta}}^{n_2\textrm{ times}}
\circ 
\overbrace{F_{\sigma^{T_1-1}\omega,\alpha} \circ \cdots \circ F_{\omega,\alpha}}^{n_1\textrm{ times}}.
\end{align}
So define an alternating sequence $(\gamma_m)_{m\geq 1}\subset \{\alpha, \beta\}$ by setting $\gamma_{2m-1}=\alpha$ and $\gamma_{2m}=\beta$ for each $m\geq 1$. We choose positive integers $n_m$ recursively and write $T_0=-$ and $T_m = n_1 + \cdots +n_m$. We define
\begin{align}\label{eq:Mbound}
M_{m-1}(\omega)\coloneq\begin{cases}
0 & m=1\\
\sup_{y\in [0, 1]}\left|\frac{F^{(T_{m-1})}_\omega(y)-y}{T_{m-1}}\right| & m\geq 2.
\end{cases}
\end{align}
Take a sequence $\varepsilon_m$ with $\lim_{m\to\infty}\varepsilon_m=0$. Choose $n_m$ so large that
\begin{equation}\label{eq:largenm}
n_m \geq N_{\gamma_m}(\sigma^{T_{m-1}}\omega, \varepsilon_m)
\end{equation}
and
\begin{align}\label{eq:Tmbound}
\frac{T_{m-1}}{T_m}(M_{m-1}(\omega)+\max\{|a|, |b|\})<\varepsilon_m.
\end{align}
This is possible because $\Delta F^{(T_{m-1})}_\omega(y)=F^{(T_{m-1})}_\omega(y)-y$ is continuous and $1$-periodic in $y$, so $M_{m-1}(\omega)$ is finite for each fixed $m$.

We apply Proposition~\ref{pro:blockcoincidence} to the sequences $(\gamma_m)_{m\geq 1}$ and $(n_m)_{m\geq 1}$ to obtain block coincidence sets $C^{(m)}_\omega$ for $m\geq 0$ and a global block coincidence set $C_\omega$. 
Take $x\coloneq x_{a, b}(\omega)\in C_\omega\cap[0,1]$.
For each $m\geq 1$, denote by $z_m\coloneq F^{(T_{m-1})}_\omega(x)$, the value of the orbit of $x$ at the start of the $m$th time block. Then $z_m\in S^{n_m}_{\sigma^{T_{m-1}}\omega, \gamma_m}$, and therefore
\begin{equation}\label{eq:coinvalue}
F^{(T_{m-1}+r)}_\omega(x)= F^{(r)}_{\sigma^{T_{m-1}}\omega, \gamma_m}(z_m) \quad \textrm{for every }0\leq r\leq n_m.
\end{equation}
We set
$$
A_m\coloneq\frac{F^{(T_m)}_\omega(x)-x}{T_m}\quad\textrm{and}\quad B_m\coloneq \frac{F^{(n_m)}_{\sigma^{T_{m-1}}\omega, \gamma_m}(z_m)-z_m}{n_m},
$$
and we have
$$
A_m=\frac{T_{m-1}}{T_m}A_{m-1}+\frac{n_m}{T_m}B_m \quad \textrm{for every }m\geq 1,
$$
with the convention $A_0=0$. By \eqref{eq:pointsup} and \eqref{eq:largenm}, we have that $|B_m-\rho(F_{\gamma_m})|<\varepsilon_m$.
By \eqref{eq:Mbound}, we have $|A_{m-1}|\leq M_{m-1}(\omega)$. Hence
\begin{align*}
|A_m-\rho(F_{\gamma_m})|&\leq \frac{T_{m-1}}{T_m}|A_{m-1}-\rho(F_{\gamma_m})|+\frac{n_{m}}{T_m}|B_{m}-\rho(F_{\gamma_m})|\\
&=\frac{T_{m-1}}{T_m}(M_{m-1}(\omega)+\max\{|a|, |b|\})+\varepsilon_m<2\varepsilon_m,
\end{align*}
by \eqref{eq:Tmbound}. Therefore,
\begin{equation}\label{eq:theboundofam}
A_{2m-1}\to a, \quad A_{2m}\to b.
\end{equation}
So
$$
\liminf_{n\to \infty}\frac{F^{(n)}_\omega(x)-x}{n}\leq a\quad\textrm{and}\quad\limsup_{n\to \infty}\frac{F^{(n)}_\omega(x)-x}{n}\geq b.
$$

For each $j\geq 0$, let $G_{j, \omega}$ denote the fibre map applied in \eqref{eq:blockcomposition} at time $j$: that is,
$G_{j, \omega}=F_{\sigma^j\omega, \gamma_m}$ for $T_{m-1}\leq j <T_m$. Let $H^{(n)}_\omega\coloneq G_{n-1, \omega}\circ \cdots \circ G_{0, \omega}$.
By \eqref{eq:coinvalue}, for each $m\geq 1$, and for each $n$ in the $m$th time block $[T_{m-1},T_m)$, we have $H^{(n)}_\omega(x)=F^{(n)}_\omega(x)$.
Therefore,
\begin{align}\label{eq:orbitcoincidence}
F^{(n)}_\omega(x)=H^{(n)}_\omega(x)\quad \textrm{for every }n\geq 0.
\end{align}
Since $\alpha\leq \gamma_m\leq \beta$ for each $m\geq 1$, by the monotonicity of the interpolating maps, we have
$F_{\sigma^j\omega,\alpha}\leq G_{j,\omega}\leq F_{\sigma^j\omega,\beta}$ for every $j\geq 0$. Since monotonicity is preserved by composition, we have $F^{(n)}_{\omega, \alpha}(x)\leq H^{(n)}_\omega(x)\leq F^{(n)}_{\omega, \beta}(x)$ for each $n\in\mathbb{N}$.
Hence
$$
\frac{F^{(n)}_{\omega, \alpha}(x)-x}{n}\leq \frac{H^{(n)}_\omega(x)-x}{n}\leq \frac{F^{(n)}_{\omega, \beta}(x)-x}{n}.
$$
Taking upper and lower limits, by \eqref{eq:orbitcoincidence} we find
$$
a=\rho(F_\alpha)\leq \liminf_{n\to \infty}\frac{F^{(n)}_{\omega}(x)-x}{n}\leq \limsup_{n\to \infty}\frac{F^{(n)}_{\omega}(x)-x}{n}\leq \rho(F_\beta)=b.
$$
Combining this with \eqref{eq:theboundofam}, we have
$$
\liminf_{n\to \infty}\frac{F^{(n)}_{\omega}(x)-x}{n}=a \quad\textrm{and}\quad \limsup_{n\to \infty}\frac{F^{(n)}_{\omega}(x)-x}{n}=b.
$$   
\end{proof}

As a consequence of Theorem \ref{thm:realupperlower}, we can show that any closed subinterval of the random rotation set can be realised almost surely as the set of accumulation points of the average displacement along a single orbit.

\begin{proof}[Proof of Corollary~\ref{cor:realsubint}]
If $a=b$, the result follows from the Proposition~\ref{pro:simrealisation}. Therefore we assume $a<b$. By Theorem~\ref{thm:realupperlower}, there exists a full measure subset $\Omega_{a, b}$ with the property that for every $\omega\in \Omega_{a, b}$, there exists a point $x_I(\omega)\in \mathbb{R}$ with
$$
\liminf_{n\to\infty}\frac{F^{(n)}_\omega(x_I(\omega))-x_I(\omega)}{n}=a \quad\textrm{and}\quad  \limsup_{n\to\infty}\frac{F^{(n)}_\omega(x_I(\omega))-x_I(\omega)}{n}=b.
$$

We now prove that the interval $[a, b]$ is realised as the set of accumulation values of these displacement averages. 
Define $\psi(\omega)\coloneq \|\Delta F_\omega\|=\sup_{y\in \mathbb{R}}|F_\omega(y)-y|$. By the integrability condition \eqref{eq:integrability}, we have $\psi \in L^1(\Omega)$. For each $n\geq 1$, we have 
$$
F^{(n+1)}_\omega(x_I(\omega))-x_I(\omega)=F^{(n)}_\omega(x_I(\omega))-x_I(\omega)+ F_{\sigma^n\omega}\left(F^{(n)}_\omega(x_I(\omega))\right)-F^{(n)}_\omega(x_I(\omega)).
$$
Set
$$
u_n=\frac{F^{(n)}_\omega(x_I(\omega))-x_I(\omega)}{n}.
$$
Then,
$$
u_{n+1}-u_n=\frac{ F_{\sigma^n\omega}\left(F^{(n)}_\omega(x_I(\omega))\right)-F^{(n)}_\omega(x_I(\omega))}{n+1}-\frac{u_n}{n+1}.
$$
Hence
$$|u_{n+1}-u_n|\leq \frac{\psi(\sigma^n\omega)}{n+1}+\frac{|u_n|}{n+1}.
$$
So we have
$$\left|\frac{F^{(n)}_\omega(x_I(\omega))-x_I(\omega)}{n}\right|=|u_n|\leq \frac{1}{n}\sum^{n-1}_{k=0}\psi(\sigma^k\omega).$$

By Birkhoff's ergodic theorem, the right-hand side is finite for almost every $\omega$. Moreover, since $\psi\in L^1(\Omega)$, by the Borel-Cantelli lemma, $\psi(\sigma^n\omega)/n\to 0$ for almost every $\omega$. Indeed, for each $\varepsilon>0$,
$$
\sum^{\infty}_{n=1}\mathbb P\{\psi(\sigma^n\omega)>\varepsilon n\}
=
\sum_{n=1}^\infty
\mathbb P\{\psi(\omega)>\varepsilon n\}
<\infty.
$$
Thus, after replacing $\Omega_{a, b}$ by a smaller full-measure set if necessary, we have $u_{n+1}-u_n\to 0$.

Now we identify the accumulation points of $(u_n)$. Since $\liminf_{n\to \infty} u_n=a$ and $\limsup_{n\to \infty} u_n=b$, every accumulation point $(u_n)$ belongs to $[a, b]$, and both $a$ and $b$ are accumulation points of $(u_n)$. Conversely, take any $c\in (a, b)$. We shall show that $c$ is an accumulation point of $(u_n)$. Let $\varepsilon>0$ be small enough such that $a<c-\varepsilon<c+\varepsilon<b$. Since $\liminf_{n\to \infty} u_n=a$, the sequence $u_n$ is below $c-\varepsilon$ infinitely often. Since $\limsup_{n\to \infty} u_n=b$, the sequence $u_n$ is above $c+\varepsilon$ infinitely often.

Choose $N$ such that $|u_{n+1}-u_n|<\varepsilon$ for all $n\geq N$. Choose $i\geq N$ such that $u_i<c-\varepsilon$. Since $u_n>c+\varepsilon$ infinitely often, let $j>i$ be the first index after $i$ such that $u_j>c+\varepsilon$. By the minimality of $j$, $u_{j-1}\leq c+\varepsilon$. Since $|u_j-u_{j-1}|< \varepsilon$, we obtain $c+\varepsilon<u_j<c+2\varepsilon$.
Since the initial index $i$ can be chosen arbitrarily large, this argument gives arbitrary large $j$ with $c+\varepsilon<u_j<c+2\varepsilon.$ Taking a sequence $ \varepsilon_k\downarrow 0$, we obtain a subsequence of $(u_n)$ converging to $c$. Hence each $c\in (a, b)$ is an accumulation point. Therefore
$\operatorname{Clust}\{u_n: n\geq 1\}=[a, b]$.
That is,
$$
\operatorname{Clust}\left\{\frac{F^{(n)}_\omega(x_I(\omega))-x_I(\omega)}{n}: n\geq 1\right\}=[a, b]=I.
$$
\end{proof}

\subsection{A Bernoulli example}\label{sec:Bernoulli}

We give an example of a random circle endomorphism and a demonstration of how to calculate the random rotation set.

\begin{example}(Calculating the random rotation set for a random piecewise linear circle endomorphism)\label{ex:Bernoulli}
Given $p\in(0,1)$, we construct a specific random circle endomorphism based on two piecewise linear circle maps $f_{+1}$ and $f_{-1}$, where the random compositions are determined by a Bernoulli process, with map $f_{+1}$ selected with probability $p$ and map $f_{-1}$ selected with probability $1-p$. More specifically, consider the one-sided sequence space $\Omega=\{+1,-1\}^{\mathbb{N}}$ and let $(\Omega, \mathcal{F}, \mathbb{P}, \sigma)$ be the one-sided $(p, 1-p)$-shift transformation on the two symbols $+1$ and $-1$. Consider the random circle endomorphism for which the lift $F\in\widetilde{E}(\Omega)$, denoted by $F(\omega)\coloneq F_{\omega_1}$, of the generator $f\in\mathcal{E}(\Omega)$ is given on the interval $[0, 1)$ by
$$
F_{-1}|_{[0, 1)}(x)=\begin{cases}
    -3x & 0\leq x< 1/3\\
    3x-2  &  1/3 \leq x <1\\
\end{cases}
\quad\textrm{and}\quad
F_{+1}|_{[0, 1)}(x)=\begin{cases}
    3x & 0 \leq x<2/3\\
    -3x+4 & 2/3 \leq x <1\\
\end{cases},
$$
from which each map is extended to the real line using the degree-one property (see Figure \ref{fig:Bernoulli}). Note that when restricted to the integers, both $F_{-1}$ and $F_{+1}$ reduce to the identity map.

\begin{figure}
\begin{center}
\includegraphics[scale=0.35]{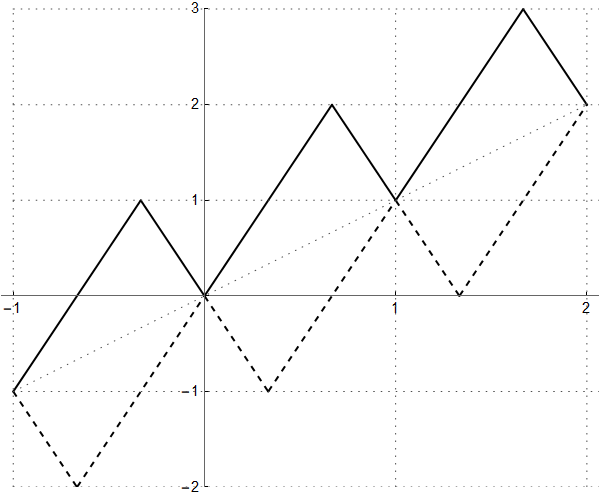}
\end{center}
\caption{Graphs of the lifted fibre maps $F_{+1}$ (solid line) and $F_{-1}$ (dashed line).}
\label{fig:Bernoulli}
\end{figure}

For each $\omega\in \Omega$, we construct a point $x=x(\omega)\in[0,1]$ in terms of a base 3 expansion as follows:
$$
x=x(\omega)=(0.x_1 x_2 x_3\cdots)_{\textrm{base}\,3}=\sum_{n=1}^\infty\frac{x_n}{3^{n}}\in [0, 1], 
$$
where the base~3 digits $x_n$ of $x(\omega)$ satisfy
$$
x_n=\begin{cases}
    1 & \textrm{ if } \omega_n=+1\\
    2 & \textrm{ if } \omega_n=-1
\end{cases}.
$$
We have $x(\omega)\in [1/3,2/3]$ if $\omega_1=+1$ and $x(\omega)\in [2/3,1]$ if $\omega_1=-1$.  Note that $F_{+1,U}|_{[1/3,2/3]}=F_{+1}|_{[1/3,2/3]}$ is an affine map onto $[1,2]$ and $F_{-1,U}|_{[1/3,2/3]}=F_{-1}|_{[1/3,2/3]}$ is an affine map onto $[0,1]$.
So we have
$$
F_\omega(x)=F_{\omega,U}(x)=(0.x_2 x_3 \cdots)_{\textrm{base}\,3}+\begin{cases}
    0& \textrm{ if } \omega_1=-1\\
    1& \textrm{ if } \omega_1=+1
\end{cases}.
$$
By induction, for each $n\in \mathbb{N}$, we have $F^{(n)}_\omega(x)=F^{(n)}_{\omega,U}(x)=(0.x_{n+1} x_{n+2} \cdots)_{\textrm{base}\,3}+K(n, \omega)$,
where $K(n, \omega)=\#\left\{1\leq i\leq n: \omega_i=+1\right\}$.  Thus we have $\lvert F^{(n)}_{\omega,U}(x)-K(n, \omega)\rvert \leq 1$
for all $n\in \mathbb{N}$. Therefore, by Theorem \ref{thm:rotset},
$$
\rho_{F_U}(\omega)=\lim_{n\to\infty}\frac{F^{(n)}_{\omega,U}(x)-x}{n}=\lim_{n\to\infty}\frac{K(n, \omega)}{n}=\lim_{n\to\infty}\frac{1}{n}\sum^{n-1}_{i=0}\chi_A(\sigma^i\omega), 
$$
for almost every $\omega\in\Omega$ where $A=\left\{\omega\in \Omega:\omega_1=+1\right\}$. Thus by Birkhoff's ergodic theorem applied to the ergodic transformation $\sigma$ with the observable function $\chi_A$, we have $\rho_{F_U}(\omega)=p$ for almost every $\omega\in\Omega$. 

Similarly, given $\omega\in \Omega$, we construct a point $x'=x'(\omega)$, in terms of a base 3 expansion, as follows:
$$
x'=x'(\omega)=(0.x_1 x_2 x_3\cdots)_{\textrm{base}\,3}=\sum_{n=1}^\infty\frac{x'_n}{3^{n}}\in [0, 1], 
$$
where the base~3 digits $x'_n$ of $x'(\omega)$ satisfy
$$
x'_n=\begin{cases}
    0 & \textrm{ if } \omega_n=+1\\
    1 & \textrm{ if } \omega_n=-1
\end{cases}.
$$
We have $x'(\omega)\in [0,1/3]$ if $\omega_1=+1$ and $x'(\omega)\in [1/3,2/3]$ if $\omega_1=-1$. Note that $F_{+1,L}|_{[0,1/3]}=F_{+1}|_{[0,1/3]}$ is an affine map onto $[0,1]$ and $F_{-1,L}|_{[0,1/3]}=F_{-1}|_{[0,1/3]}$ is an affine map onto $[-1,0]$.
So we have
$$
F_\omega(x')=F_{\omega,L}(x')=(0.x_2 x_3 \cdots)_{\textrm{base}\,3}+\begin{cases}
    -1& \textrm{ if } \omega_1=-1\\
    0& \textrm{ if } \omega_1=+1
\end{cases}.
$$
Following similar steps with before, now using the characteristic function of the set $\{\omega\in \Omega:\omega_n=-1\}$, we obtain
$$
\rho_{F_L}(\omega)=\lim_{n\to \infty}\frac{F^{(n)}_{\omega,L}(x')-x'}{n}=-(1-p) \quad \textrm{for a.e.\ }\omega\in \Omega.
$$
Hence, $\rho(F)=[-(1-p),p]$.
\end{example}

\subsection{Properties of random rotation sets}\label{sec:proprotset}
In this section, we introduce some elementary properties of the random rotation set which are used to explain its independence with respect to the choice of the lift and its behaviour with respect to order and its cohomological equivalence.

\begin{proposition}(Invariance of length of random rotation set)\label{pro:lengthrotset}
Let $(\Omega, \mathcal{F}, \mathbb{P}, \sigma)$ be a measure-preserving dynamical system, let $f\in \mathcal{E}(\Omega)$, and let $F,G\in \widetilde{\mathcal{E}}(\Omega)$ be random lifts of $f$. Then there exists a $\sigma$-invariant full measure set $\Omega_0\subset \Omega$ such that 
$$
|\mathrm{Rot}_F(\omega)|=|\mathrm{Rot}_G(\omega)|\quad  \textrm{for a.e. }\omega\in \Omega_0.
$$
\end{proposition}
\begin{proof}
Let $N(\omega)$ be defined for each $\omega\in\Omega$ by $N(\omega)\coloneq G_\omega-F_\omega$. As $F_\omega$ and $G_\omega$ are lifts of the same map of the circle, their difference $N(\omega)$ is an integer. Since $F, G\in \widetilde{\mathcal{E}}(\Omega)$, we have $N\in L^1(\Omega)$. Thus we have an integrable integer-valued function $N:\Omega\to\mathbb{Z}$.
By repeated use of the degree one property \eqref{eq:degree1}, 
$$
G^{(n)}_\omega(x)= F^{(n)}_\omega(x)+\sum_{i=0}^{n-1} N(\sigma^i\omega).
$$
So we have
\begin{equation}\label{eq:tworandomlifts}
\frac{G^{(n)}_\omega(x)-x}{n}=\frac{F^{(n)}_\omega(x)-x}{n}+\frac{\sum^{n-1}_{i=0}N(\sigma^i\omega)}{n}.
\end{equation}
Since $N$ is integrable, by Birkhoff's ergodic theorem there exists a $\sigma$-invariant full measure subset $\Omega_0$ such that for every $\omega\in \Omega_0$,
$$
N^*(\omega)=\lim_{n\to \infty}\frac{1}{n}\sum^{n-1}_{i=0}N(\sigma^i\omega)
$$
exists. By \eqref{eq:tworandomlifts}, for every $\omega\in \Omega_0$ and any $x\in \mathbb{R}$, we have
\begin{equation}\label{eq:limsuptworandomlifts}
\limsup_{n\to\infty}\frac{G^{(n)}_\omega(x)-x}{n}=\limsup_{n\to \infty}\frac{F^{(n)}_\omega(x)}{n}+N^*(\omega)
\end{equation}
By \eqref{eq:limsuptworandomlifts}, for each $\omega\in \Omega_0$ we have $\mathrm{Rot}_G^+(\omega)=\mathrm{Rot}_F^+(\omega)+N^*(\omega)$
and so by taking the closure, we also have $\mathrm{Rot}_G(\omega)=\mathrm{Rot}_F(\omega)+N^*(\omega)$.
Hence the length of the random rotation set is independent of the choice of random lift for almost every $\omega\in\Omega$.
\end{proof}

Thus, although the position of the random rotation set depends on the choice of random lift, its length does not. The next lemma demonstrates the monotonicity of the random rotation set with respect to the pointwise order on random lifts.

\begin{lemma}[Monotonicity of the random rotation set]\label{lem:monotonicityrotset}
Let $f$ be a random circle endomorphism and let $F, G\in\widetilde{\mathcal{E}}(\Omega)$ be two random lifts of $f$. If $F_\omega(x)\leq G_\omega(x)$ for all $x\in \mathbb{R}$ and almost every $\omega\in\Omega$, then we have
$$
\mathrm{Rot}_F(\omega)\preceq \mathrm{Rot}_G(\omega) \quad \textrm{for a.e.\ }\omega\in \Omega, 
$$
where, for compact intervals $I, J \subset \mathbb{R}$, $I \preceq J$ means that 
$\inf I \leq \inf J$ and $\sup I \leq \sup J$.
\end{lemma}
\begin{proof}
Since $F_\omega(x)\leq G_\omega(x)$ for all $x\in \mathbb{R}$ and $\omega\in \Omega$, the upper and lower maps satisfy
$F_{\omega,L}(x)\le G_{\omega,L}(x)$ and $F_{\omega,U}(x)\le G_{\omega,U}(x)$.
By the monotonicity of the random rotation number for random monotone circle maps, we have
$\rho_{F_L}(\omega)\leq \rho_{G_L}(\omega)$ and $\rho_{F_U}(\omega)\le \rho_{G_U}(\omega)$.
On the other hand, by Theorem~\ref{thm:rotset}, on a full measure subset $\Omega'_F\subset\Omega$, we have 
$\inf\mathrm{Rot}_F(\omega)=\rho_{F_L}(\omega)$ and $\sup\mathrm{Rot}_F(\omega)=\rho_{F_U}(\omega)$ and likewise on a full measure subset $\Omega'_G\subset\Omega$ we have $\inf\mathrm{Rot}_G(\omega)=\rho_{G_L}(\omega)$ and $\sup\mathrm{Rot}_G(\omega)=\rho_{G_U}(\omega)$.
Hence $\inf\mathrm{Rot}_F(\omega)\leq \inf\mathrm{Rot}_G(\omega)$ and $\sup\mathrm{Rot}_F(\omega)\leq \sup\mathrm{Rot}_G(\omega)$ for all $\omega\in\Omega'_F\cap\Omega'_G$.
Therefore, for almost every $\omega\in\Omega$, we have $\mathrm{Rot}_F(\omega)\preceq\mathrm{Rot}_G(\omega)$.
\end{proof}

The next property concerns the invariance of the random rotation set under cohomological equivalence. We firstly recall the definition of cohomological equivalence used here.

\begin{definition}[Cohomological equivalence]\label{def:cohomequiv}
Given random lifts $F, G\in \widetilde{\mathcal{E}}(\Omega)$ of generators of random circle endomorphisms over a common base system $(\omega,\mathcal{F},\mathbb{P},\sigma)$ and a lift $H\in\widetilde{H}(\Omega)$ of a random circle homeomorphism, we say $F$ and $G$ are \emph{cohomologous via $H$}, written $F\sim_H G$, if 
\begin{align}
H_{\sigma\omega}\circ F_\omega=G_\omega \circ H_\omega \quad \textrm{for a.e. }\omega\in \Omega.
\end{align}
\end{definition}

We now show that if random monotone circle maps over the same base transformation have random lifts that are cohomologically equivalent, then they have the same random rotation number. 

\begin{lemma}(Cohomological invariance for random monotone maps)\label{lem:cohomequiv}
Suppose that $F, G\in \widetilde{\mathcal{M}}(\Omega)$ and $H\in \widetilde{\mathcal{H}}(\Omega)$ satisfy $F\sim_H G$, then
$$
\rho_F(\omega)=\rho_G(\omega), \quad \textrm{for a.e. }\omega\in \Omega.
$$
\end{lemma}
\begin{proof}
Given $\omega\in \Omega, n\in \mathbb{N}$, we have
\begin{align*}
|G^{(n)}_\omega(0)-F^{(n)}_\omega(0)|&=|G^{(n)}_\omega(0)-G^{(n)}_\omega(H_\omega(0))|+|H_{\sigma^n\omega}(F^{(n)}_\omega(0))-
F^{(n)}_\omega(0)|\\
 &\leq |G^{(n)}_\omega(\mathrm{frac}(H_\omega(0))+\lfloor H_{\omega}(0)\rfloor)-G^{(n)}_\omega(0)|+|H_{\sigma^n\omega}(F^{(n)}_\omega(0))-F^{(n)}_\omega(0)|\\
 &\leq |G^{(n)}_\omega(\mathrm{frac}(H_\omega(0)))-G^{(n)}_\omega(0)|+|\lfloor H_{\omega}(0)\rfloor|+|H_{\sigma^n\omega}(F^{(n)}_\omega(0))-F^{(n)}_\omega(0)|\\
 &\leq 1+|\lfloor H_\omega(0)-0\rfloor|+ \lVert \Delta_{H_{\sigma^n\omega}}\rVert \\
 &\leq 1+(1+ \lVert \Delta_{H_\omega}\rVert)+\lVert \Delta_{H_{\sigma^n\omega}}\rVert,
\end{align*}
where $\lfloor\cdot\rfloor$ denotes the floor function. So
\begin{align*}
\frac{1}{n}\lvert G^{(n)}_\omega(0)-F^{(n)}_\omega(0)\rvert &\leq \frac{2}{n}+\frac{\lVert H_\omega \rVert}{n}+\frac{\lVert H_{\sigma^n\omega} \rVert}{n}
\end{align*}
which tends to zero for almost every $\omega\in\Omega$. This is because the final term satisfies
\begin{align*}
\frac{1}{n}\lVert \Delta_{H_{\sigma^n\omega}}\rVert &=\frac{1}{n+1}\sum_{k=0}^{n} \lVert \Delta_{H_{\sigma^k\omega}}\rVert-\frac{n}{n+1}\times \frac{1}{n}\sum^{n-1}_{k=0}\lVert \Delta_{H_{\sigma^k\omega}}\rVert,
\end{align*}
which tends to zero as $n\to\infty$ for almost every $\omega\in\Omega$, by Birkhoff's ergodic theorem. Hence 
$$
\rho_F(\omega)=\lim_{n\rightarrow \infty}\frac{F^{(n)}_\omega(0)}{n}=\lim_{n\rightarrow \infty}\frac{G^{(n)}_\omega(0)}{n}=\rho_G(\omega),\quad  \textrm{for a.e.\ }\omega \in \Omega.
$$
\end{proof}

We obtain the corresponding result for random endomorphisms by applying the previous result to the upper and lower random maps.

\begin{theorem}(Cohomological invariance)\label{thm:cohominvariance}
Suppose that $F, G\in \widetilde{\mathcal{E}}(\Omega)$ are random lifts of generators of random circle endomorphisms over the same base transformation $(\Omega,\mathcal{F},\mathbb{P},\sigma)$ that are cohomologically equivalent via $H\in \widetilde{\mathcal{H}}(\Omega)$: that is, $F\sim_H G$. Then
\begin{align}
\mathrm{Rot}_F(\omega)=\mathrm{Rot}_G(\omega) \quad \textrm{for a.e. }\omega\in \Omega.
\end{align}
\end{theorem}
\begin{proof}
Since $F\sim_H G$, we have $H_{\sigma\omega}(F_\omega(x))=G_\omega(H_\omega(x))$ for all $x\in \mathbb{R}$ and almost every $\omega\in \Omega$.
Taking the infimum over values of $y$ greater than or equal to $x$, we have
\begin{align*}
\inf_{y\geq x}H_{\sigma\omega}(F_\omega(y))=\inf_{y\geq x}G_\omega(H_\omega(y)).
\end{align*}
Since $H$ is a lift of a random circle homeomorphism, the map $H_{\sigma\omega}$ is strictly increasing and so we have
\begin{align*}
H_{\sigma\omega}(F_{\omega, L}(x))=H_{\sigma\omega}\left(\inf_{y\geq x} (F_\omega(y)\right)=\inf_{y\geq x}G_\omega(H_\omega(y))=G_{\omega, L}(H_\omega(x)).
\end{align*}
The argument is similar for $F_U$ and $G_U$. Hence we have $F_L \sim_H G_L$ and $F_U\sim_H G_U$. By Lemma~\ref{lem:cohomequiv}, we have for almost every $\omega\in\Omega$ that $\rho_{F_L}(\omega)=\rho_{G_L}(\omega)$ and $\rho_{F_L}(\omega)=\rho_{G_L}(\omega)$. Therefore, by Theorem~\ref{thm:rotset}, $\mathrm{Rot}_F(\omega)=\mathrm{Rot}_G(\omega)$ for almost every $\omega\in \Omega$.
\end{proof}

\section{Random topological entropy for random circle endomorphisms}\label{sec:randomtopo}

In this section, we establish a relationship between the non-triviality of the random rotation set and the topological complexity of the system.

\subsection{Random topological entropy}

Firstly, we briefly recall the metric framework for random topological entropy. Although we use the concept of lifts $F_\omega$ defined on $\mathbb{R}$, dynamical separation is intrinsically measured on the compact quotient space $\mathbb{S}^1$. Let $d_{\mathbb{S}^1}$ denote the standard metric on the circle. For a random circle endomorphism $f \in \mathcal{E}(\Omega)$, the sample-path dependent \emph{Bowen metric over $n \in \mathbb{N}$ steps} is defined for $x, y \in \mathbb{S}^1$ by
$$ 
d_{\omega, n}^f(x, y) \coloneq \max_{0 \le k < n} d_{\mathbb{S}^1}\bigl(f^{(k)}_\omega(x), \, f^{(k)}_\omega(y)\bigr). 
$$
For any precision $\varepsilon > 0$, a subset $E \subset \mathbb{S}^1$ is an \emph{$(n, \varepsilon, \omega)$-separated set} if $d_{\omega, n}^f(x, y) \ge \varepsilon$ for all distinct $x, y \in E$. Letting $s(n, \varepsilon, \omega, f)$ denote the maximum cardinality of such a set, the exponential growth rate at scale $\varepsilon$ is
$$ 
h_\varepsilon(f, \omega) \coloneq \limsup_{n \to \infty} \frac{1}{n} \log s(n, \varepsilon, \omega, f). 
$$
It follows from the standard theory of random topological entropy \cite{Bogenschutz1993, Kifer1986} that the function $\omega \mapsto h_\varepsilon(f, \omega)$ is measurable. Moreover, the cocycle property implies that it is $\sigma$-invariant. If the base system $(\Omega, \mathcal{F}, \mathbb{P}, \sigma)$ is ergodic, then $h_\varepsilon(f, \omega)$ assumes a constant value $\mathbb{P}$-almost surely. The \emph{random topological entropy} of $f$ is then defined as the almost sure limit $h_{\mathrm{top}}^{(\mathrm{r})}(f) \coloneq \lim_{\varepsilon \to 0} h_\varepsilon(f, \omega)$.

We prove that for random monotone circle maps, the random topological entropy is zero.

\begin{proof}[Proof of Proposition~\ref{pro:zeroentropy}]
Fix $\omega \in \Omega$, an integer $n \ge 1$, and precision $0 < \varepsilon < 1/2$. Let $E = \{x_1, x_2, \dots, x_m\} \subset \mathbb{S}^1$ be an $(n, \varepsilon, \omega)$-separated set, indexed in cyclic order around the circle. If $m\leq 1$, the desired estimate is trivial, so we assume $m \geq 2$. Choose lifts $x_1<x_2<\cdots<x_m<x_1+1$  of $x_1, \cdots, x_m$ and set $x_{m+1}:=x_1+1$. These points determine the adjacent arcs $I_i=[x_i, x_{i+1})$ on $\mathbb{S}^1$. 
Take a lift $F\in \widetilde{\mathcal{M}}(\Omega)$ of $f$. Since $F_\omega$ is monotone and degree one for each $\omega\in\Omega$, the same is true of each composition $F^{(j)}_\omega$, $j\geq 1$. For $0\leq j<n$, define the oriented length of the image of $I_i$ at time $j$ by $\ell_j(I_i):=
F_\omega^{(j)}(x_{i+1})-F_\omega^{(j)}(x_i).$ Since $F^{(j)}_\omega$ is monotone increasing and degree one, we have $\ell_j(I_i)\geq 0$ and 

\begin{equation}\label{eq:sumlengths}
\sum_{i=1}^m \ell_j(I_i)=F_\omega^{(j)}(x_1+1)-F_\omega^{(j)}(x_1)=1
\qquad \text{for every } 0 \leq j < n.
\end{equation}

Since $E$ is an $(n, \varepsilon, \omega)$-separated set, any adjacent pair $x_i$ and $x_{i+1}$ must separate by at least $\varepsilon$ at some time before $n$. That is, for each $i \in \{1, \dots, m\}$, there exists $j_i \in \{0, 1, \dots, n-1\}$ such that  
$$
d_{\mathbb{S}^1}\bigl(f_\omega^{(j_i)}(x_i), \, f_\omega^{(j_i)}(x_{i+1})\bigr) \ge \varepsilon.
$$
The circular distance between $f_{\omega}^{(j_i)}(x_i)$ and $f_\omega^{(j_i)}(x_{i+1})$ is bounded above by the oriented length $\ell_{j_i}(I_i)$. Hence $\ell_{j_i}(I_i)\geq \varepsilon$.

Let $N_j \coloneq \#\{i : j_i = j\}$ denote the number of arcs achieving their separation at time step $j$. We have $\sum_{j=0}^{n-1} N_j = m$. For a fixed $j$, equation \eqref{eq:sumlengths} gives 
$\sum_{i:j_i=j}\ell_j(I_i)\leq \sum_{i=1}^m \ell_j(I_i)=1.$ Since each term in the left-hand sum is at least $\varepsilon$, we obtain $N_j\varepsilon\leq 1$. Because each such arc has length $\geq \varepsilon$, we obtain $N_j \cdot \varepsilon \le 1$, which implies $N_j \leq \lfloor 1/\varepsilon \rfloor$.
Summing this bound over all time steps $0 \le j < n$ yields:
$$
m = \sum_{j=0}^{n-1} N_j \le \sum_{j=0}^{n-1} \left\lfloor \frac{1}{\varepsilon} \right\rfloor = n \left\lfloor \frac{1}{\varepsilon} \right\rfloor. 
$$
Since $E$ was an arbitrary separated set, the maximal cardinality satisfies $s(n, \varepsilon, \omega, f) \le n \lfloor 1/\varepsilon \rfloor$. Evaluating the exponential growth rate yields:
$$ 
h_\varepsilon(f, \omega) = \limsup_{n \to \infty} \frac{1}{n} \log s(n, \varepsilon, \omega, f) \le \limsup_{n \to \infty} \frac{\log n + \log\lfloor 1/\varepsilon \rfloor}{n} = 0. 
$$
On the other hand, $h_\varepsilon(f,\omega)\geq 0$. Hence $h_\varepsilon(f,\omega)=0$ for every $0<\varepsilon<1/2$. Letting $\varepsilon\downarrow 0$, we obtain $h_{\mathrm{top}}^{\mathrm{(r)}}(f, \omega) = 0$.
\end{proof}

\subsection{Proof of Theorem~\ref{thm:positivetopentropy}}
We now prove that a non-degenerate random rotation interval forces positive random topological entropy. The proof passes to a finite cover of the circle, where the positive length of the random rotation set yields full-cover time blocks along a positive-density set of sample times. We first show that passing to this finite cover does not change the random topological entropy.

\begin{lemma}(Finite covers preserve random topological entropy)\label{lem:finitecover}
Let $\hat{\pi}: T^1_4=\mathbb{R}/4\mathbb{Z}\to \mathbb{S}^1=\mathbb{R}/\mathbb{Z}$ be the natural covering map. Let $f$ be a random circle endomorphism and let $\hat{f}$ be the map induced by the lift $F$ of $f$ on $T^1_4$. Then
$$
h^{(\mathrm{r})}_{\mathrm{top}}(\hat{f})=h^{(\mathrm{r})}_{\mathrm{top}}(f).
$$
\end{lemma}
\begin{proof}
The argument is the random analogue of the finite-cover invariance used by \cite{Glendinningetal2009}. Fix $\omega\in \Omega$. Let $d$ and $\hat{d}$ be compatible metrics on $\mathbb{S}^1$ and $T^1_4$, and let $d_{\omega, n}$, $\hat{d}_{\omega, n}$ be the associated sample Bowen metrics. 

Since $\hat{\pi}$ is a four-fold covering, there exists $\varepsilon_0$, the restriction of $\hat{\pi}$ to any $\hat d$-ball of radius $\varepsilon$ is an isometry onto its image. Hence the projection of an $(\hat f,\omega,n,\varepsilon)$-ball is contained in an
$(f,\omega,n,\varepsilon)$-ball, and therefore $R_f(\omega,n,\varepsilon)\le R_{\hat f}(\omega,n,\varepsilon)$.
Conversely, each $(f, \omega, n, \epsilon)$-ball lifts to at most four $(\hat{f}, \omega, n, \epsilon)$-balls, because $\hat{\pi}$ is a four-fold covering. Therefore
$R_{\hat{f}}(\omega, n, \epsilon)\leq 4 R_f(\omega, n, \epsilon)$.
Thus, for all sufficiently small $\varepsilon$ and each $n\in\mathbb{N}$, $R_f(\omega,n,\varepsilon)\le R_{\hat f}(\omega,n,\varepsilon)\le 4\,R_f(\omega,n,\varepsilon)$.
Dividing by $n$ and taking the logarithm, we note that the multiplicative factor $4$ disappears when taking the limit superior. Thus $h_\varepsilon(\hat f,\omega)=h_\varepsilon(f,\omega)$.
Finally, letting $\varepsilon\to 0$ yields $h^{(\mathrm{r})}_{\mathrm{top}}(\hat{f}, \omega)=h^{(\mathrm{r})}_{\mathrm{top}}(f, \omega)$ for almost every $\omega\in\Omega$.
Hence, $h^{(\mathrm{r})}_{\mathrm{top}}(\hat{f})=h^{(\mathrm{r})}_{\mathrm{top}}(f)$.
\end{proof}

We now prove Theorem \ref{thm:positivetopentropy}.

\begin{proof}[Proof of Theorem \ref{thm:positivetopentropy}]
Let $\delta\coloneq \rho(F_U)-\rho(F_L)>0$. We define $D_n(\omega)\coloneq\operatorname{diam} (F^{(n)}_\omega([0, 1]))=F^{(n)}_{\omega, U}(1)-F^{(n)}_{\omega, L}(0)=F^{(n)}_{\omega, U}(0)+1-F^{(n)}_{\omega, L}(0)$. Since $F_L,F_U\in\widetilde{\mathcal{M}}(\Omega)$, by \eqref{eq:LiLu} the lower and upper random maps satisfy
$$
\lim_{n\to\infty}\frac{F_{\omega,L}^{(n)}(0)}{n}
=\rho(F_L)
\quad\textrm{and}\quad
\lim_{n\to\infty}\frac{F_{\omega,U}^{(n)}(0)}{n}
=\rho(F_U)\quad \textrm{for a.e. }\omega\in \Omega.
$$
So
$$
\frac{D_n(\omega)}{n}=\frac{F^{(n)}_{\omega, U}(0)}{n}-\frac{F^{(n)}_{\omega, L}(0)}{n}+\frac{1}{n}\to \rho(F_U)-\rho(F_L)=\delta>0\quad \textrm{for a.e.\ }\omega\in \Omega.
$$
Hence, there exists $N_*(\omega)$ such that when $n\geq N_*(\omega)$, $D_n(\omega)>4$. Equivalently, 
$\mathbf{1}_{\{D_n>4\}}(\omega)\to 1$ for almost every $\omega\in \Omega$. By dominated convergence, $\mathbb{P}(D_n>4)\to 1$.
Thus there exists a fixed integer $N\in \mathbb{N}$ such that the set
$A\coloneq A_N\coloneq\{\omega\in \Omega: D_N(\omega)>4\}$ satisfies $\mathbb{P}(A)>0$.

Let $T^1_4\coloneq \mathbb{R}/4\mathbb{Z}$ and let $\hat{f}$ denote the random circle endomorphism on $T^1_4$ induced by the random lift $F$.
Let $I_j \coloneq [j-1, j] \subset T^1_4$, for $j=1,2,3,4$. If $\omega \in A$, then $D_N(\omega)>4$, and so $F^{(N)}_\omega([0, 1])\subset \mathbb{R}$ is a connected interval of length exceeding 4. Hence its image modulo 4 is the whole of $T^1_4$. Since $F^{(N)}_\omega$ is a degree one map, 
$F^{(N)}_\omega([j-1, j])=F^{(N)}_\omega([0, 1])+(j-1)$, and therefore
$$
\hat{f}^{(N)}_\omega(I_j)=T^1_4, \quad \textrm{for } j=1, 2, 3, 4 \textrm{ and every }\omega\in A.
$$
In particular, this holds for the two disjoint intervals
$I_1=[0, 1]$ and $I_3=[2, 3]$, whose distance in $T^1_4$ is $\eta_0\coloneq\operatorname{dist}(I_1, I_3)>0$.
By Birkhoff's ergodic theorem applied to the indicator function $\mathbf{1}_A$, there exists a full-measure $\sigma$-invariant set $\Omega_0\subset \Omega$ such that for every $\omega\in \Omega_0$,
$$
\frac{D_n(\omega)}{n}\to \delta, \quad\textrm{and}\quad \frac{1}{m}\sum^{m-1}_{j=0}\mathbf{1}_A(\sigma^j \omega)\to \mathbb{P}(A).
$$
Hence, for every $\omega\in \Omega_0$ the set $T(\omega)\coloneq\{j\geq 0: \sigma^j\omega \in A\}$ satisfies
$$
\lim_{m\to \infty}\frac{\#(T(\omega)\cap [0, m-1])}{m}=\mathbb{P}(A)>0.
$$
Fix $\omega\in \Omega_0$. Since $T(\omega)$ has positive density, it is non-empty. Moreover, since $h_\epsilon(\hat{f}, \omega)$ is $\sigma$-invariant for each $\epsilon>0$, we may replace $\omega$ by a forward iterate if necessary. Thus, without loss of generality, we assume that $\omega\in A$, so that $0\in T(\omega)$.
Now recursively define a sparse subsequence $0=s_0(\omega)< s_1(\omega)<s_2(\omega)< \cdots$ by taking $s_0(\omega)=0$. When $s_k(\omega)$ is defined, let $s_{k+1}(\omega)$ be the smallest element of $T(\omega)$ such that $s_{k+1}(\omega)\geq s_k(\omega)+N$.
Thus the intervals $[s_k(\omega), s_k(\omega)+N)$ are pairwise disjoint.

If $s_k(\omega)\leq n <s_{k+1}(\omega)$ and $n\in T(\omega)$, then by minimality of $s_{k+1}(\omega)$ we must have $n<s_k(\omega)+N$. Therefore, for every $R>0$,
$$
T(\omega)\cap [0, R)\ \subset \bigcup_{k: s_k(\omega)<R}\left(\, [s_k(\omega), s_k(\omega)+N)\cap\mathbb{Z}\,\right).
$$
Since each half-open interval $[s_k(\omega), s_k(\omega)+N)$ contains at most $N$ integers, it follows that
$$
\#(T(\omega)\cap [0, R))\leq N\cdot \#\{k: s_k(\omega)<R\}.
$$
Dividing by $R$ and taking the limit inferior, we obtain
$$
\liminf_{R\to \infty}\frac{\#\{k:s_k(\omega)<R\}}{R}\geq \frac{\mathbb{P}(A)}{N}>0.
$$
For each $k\geq 0$, we define the block map
$$
\hat{g}_{k, \omega}\coloneq\hat{f}^{s_{k+1}(\omega)-s_k(\omega)}_{\sigma^{s_k(\omega)}\omega}=\hat{f}^{(\ell_k)}_{\sigma^{s_k(\omega)+N}\omega} \circ\hat{f}^{(N)}_{\sigma^{s_k(\omega)}\omega},
$$
where $\ell_k \coloneq s_{k+1}(\omega)-s_k(\omega)-N\geq 0$. Since $\sigma^{s_k(\omega)}\omega\in A$, we have
$\hat{f}^{(N)}_{\sigma^{s_k(\omega)}\omega}(I_j)=T^1_4$ for $j=1,2,3,4$.
All compositions of fibre maps are degree one circle endomorphisms, and so are surjective on $T^1_4$. Thus
$$
\hat{g}_{k,\omega}(I_j)=T^1_4, \quad \textrm{for } j=1, 3 \textrm{ and all }k\geq 0.
$$
Fix $q\geq 0$ and a word $\alpha=(\alpha_0, \ldots, \alpha_q)\in \{1, 3\}^{q+1}$. Define compact sets recursively by $C_q(\alpha, \omega) \coloneq I_{\alpha_q}$, and for $r=q-1, q-2, \ldots, 0$,
$$
C_r(\alpha, \omega)\coloneq I_{\alpha_r}\cap \hat{g}^{-1}_{r, \omega}(C_{r+1}).
$$
Since $\hat{g}_{r, \omega}(I_{\alpha_r})=T^1_4$, each $C_r$ is a non-empty compact subset of $I_{\alpha_r}$. In particular,
$K_\alpha(\omega)\coloneq C_0(\omega)\neq \emptyset$. Choose a point $x_\alpha\in K_\alpha(\omega)$ for each word $\alpha\in \{1, 3\}^{q+1}$.
If $\alpha \neq \beta$, let $r$ be the first index where they differ. Then $\hat{f}_\omega^{(s_r(\omega))}(x_\alpha)\in I_{\alpha_r}$ and $\hat{f}_\omega^{(s_r(\omega))}(x_\beta)\in I_{\beta_r}$.
Since $\{\alpha_r, \beta_r\}=\{1, 3\}$, these images lie in $I_1$ and $I_3$ respectively. So
$\hat{d}_{\omega, s_q(\omega)+1}(x_\alpha, x_\beta)\geq \eta_0$.
Hence the set $E_q(\omega)\coloneq\{x_\alpha: \alpha\in \{1, 3\}^{q+1}\}$ is $(s_q(\omega)+1, \eta_0, \omega)$-separated and has cardinality $2^{q+1}$. Therefore the maximum cardinality of a $(s_q(\omega)+1, \eta_0,\omega)$-separated set for $\hat{f}$ satisfies
\begin{equation}\label{eq:sfbound}
s(\omega, s_q(\omega)+1, \eta_0,\hat{f})\geq 2^{q+1}.
\end{equation}
Now, we set $N_\omega(R)\coloneq\#\{k: s_k(\omega)<R\}$. Since $N_\omega(s_q(\omega)+1)=q+1$, the density estimate gives
$$\liminf_{q\to \infty}\frac{q+1}{s_q(\omega)+1}=\liminf_{q\to \infty}\frac{N_\omega(s_q(\omega)+1)}{s_q(\omega)+1} \geq \frac{\mathbb{P}(A)}{N}.$$
Together with \eqref{eq:sfbound}, we have
$$h_{\mathrm{top}}^{\mathrm{(r)}}(\hat f,\omega)
\ge
\limsup_{q\to\infty}
\frac{1}{s_q(\omega)+1}\log 2^{q+1}
=
\log 2\cdot
\limsup_{q\to\infty}\frac{q+1}{s_q(\omega)+1}.$$
Hence 
$$h_{\mathrm{top}}^{\mathrm{(r)}}(\hat f,\omega)\geq \log 2\cdot
\limsup_{q\to\infty}\frac{q+1}{s_q(\omega)+1}\geq \log 2\cdot
\liminf_{q\to\infty}\frac{q+1}{s_q(\omega)+1} \geq \frac{\mathbb P(A)}{N} \log 2>0.$$
By Lemma~\ref{lem:finitecover}, we have $h_{\mathrm{top}}^{\mathrm{(r)}}(f,\omega) =h_{\mathrm{top}}^{\mathrm{(r)}}(\hat f,\omega)$ for almost every $\omega\in \Omega$. Hence, $h^{(\mathrm{r})}_{\mathrm{top}}(f, \omega)>0$ for almost every $\omega\in \Omega$,
and since $\sigma$ is ergodic, the random topological entropy is almost surely constant. Therefore,
$h^{(\mathrm{r})}_{\mathrm{top}}(f)>0$.
\end{proof}

\begin{remark}
This above argument should be compared with the entropy proof in the uniquely ergodically-forced setting, where the construction of a fixed finite-time full-cover follows from the topological and uniquely ergodic structure. In the random setting, the full cover is obtained only on a set $A\subset \Omega$ of positive measure. The ergodicity of the base is then used to recover a positive density family of such blocks along almost every sample path.
\end{remark}

\begin{example}[Calculation of random topological entropy]\label{ex:calcentropy}
For the Bernoulli random circle endomorphism of Example \ref{ex:Bernoulli}, we have $h_{\mathrm{top}}^{\mathrm{(r)}}(f,\omega)=\log 3$ for every $\omega\in \Omega$. In particular, $h_{\mathrm{top}}^{\mathrm{(r)}}(f)=\log 3$.

\begin{proof}
Let $I_0=[0,1/3)$, $I_1=[1/3,2/3)$ and $I_2=[2/3,1)$. Then for each $\xi\in\{+1,-1\}$ and each $j\in\{0,1,2\}$, the restriction $f_\xi|_{I_j}: I_j \to \mathbb{S}^1$ is an affine homeomorphism onto the whole circle, and its slope has absolute value $3$.

We use Bowen metrics $d^f_{\omega, n}$ and numbers $s(n, \varepsilon, \omega, f)$ introduced at the beginning of this section.

We first prove the upper bound. 
Fix $0<\varepsilon<1/2$ and choose an integer $m_\varepsilon\ge 1$ such that $3^{-m_\varepsilon}<\varepsilon$. Let $n\ge 1$ and $\omega\in\Omega$. For any word
$a=(a_0,\dots,a_{n+m_\varepsilon-1})\in\{0,1,2\}^{\,n+m_\varepsilon}$, define the corresponding cylinder interval by
$$
C_a(\omega)\coloneq\bigcap_{k=0}^{n+m_\varepsilon-1}\bigl(f_\omega^{(k)}\bigr)^{-1}(I_{a_k}).
$$
Since each fibre map has exactly three full affine branches with respect to the common partition $\{I_0, I_1, I_1\}$, the level-$(n+m_\varepsilon)$ cylinders are in one to one correspondence with the words in $\{0, 1, 2\}^{n+m_\varepsilon}$, and each $C_a(\omega)$ is a non-empty interval. Moreover, each inverse branch contracts by the factor $1/3$, therefore
$\operatorname{diam}C_a(\omega)=3^{-(n+m_\varepsilon)}$.
For every $0\le k<n$, the image of $f^{(k)}_\omega(C_a(\omega))$ is an interval and 
$$
\operatorname{diam}(f_\omega^{(k)}(C_a(\omega)))=3^{-(n+m_\varepsilon-k)}\leq 3^{-m_\varepsilon}<\varepsilon.
$$

Hence every cylinder $C_a(\omega)$ has $(n, \varepsilon, \omega)$-Bowen diameter strictly smaller than $\varepsilon$. Therefore any $(n,\varepsilon,\omega)$-separated set contains at most one point in each cylinder. Since the number of cylinders is $3^{n+m_\varepsilon}$, we obtain $s(n,\varepsilon,\omega,f)\leq 3^{n+m_\varepsilon}$. Consequently,
$$
\limsup_{n\to\infty}\frac{1}{n}\log s(n,\varepsilon,\omega,f)\le \log 3.
$$
Since $\varepsilon>0$ was arbitrary, this yields
$h_{\mathrm{top}}^{\mathrm{(r)}}(f,\omega)\le \log 3$.

We now prove the lower bound. Choose compact intervals $J_0\subset I_0$, $J_1\subset I_1$ and  $J_2\subset I_2$ with pairwise positive distance, and let $\eta_0\coloneq\min_{i\neq j}\operatorname{dist}(J_i,J_j)>0$.
Fix $\omega\in\Omega$ and $n\geq 1$. For each word $a=(a_0,\dots,a_{n-1})\in\{0,1,2\}^{\,n}$, we define non-empty compact intervals $K_r(a, \omega)$ for $0\leq r<n-1$ as
$K_{n-1}(a,\omega)\coloneq J_{a_{n-1}}$. Let $K_r(a,\omega)$ be the unique compact interval contained in $I_{a_r}$ such that $f_{\sigma^r\omega}\bigl(K_r(a,\omega)\bigr)=K_{r+1}(a,\omega)$.
This is well-defined because $f_{\sigma^r\omega}|_{I_{a_r}}: I_{a_r}\mapsto \mathbb{S}^1$ is an affine homeomorphism from $I_{a_r}$ onto $\mathbb{S}^1$. In particular, $K_0(a,\omega)\neq\emptyset$. Choose one point $x_a\in K_0(a,\omega)$ for each word $a\in\{0,1,2\}^{\,n}$.

Now let $a\neq b$, and let $r$ be the first index at which they differ. Then
$$
f_\omega^{(r)}(x_a)\in K_r(a,\omega)\subset J_{a_r}
\quad\textrm{and}\quad
f_\omega^{(r)}(x_b)\in K_r(b,\omega)\subset J_{b_r}.
$$
Since $a_r\neq b_r$, the two points lie in two distinct intervals among $J_0,J_1,J_2$, and hence $d^f_{\omega,n}(x_a,x_b)\ge \eta_0$. Therefore the set $E_n(\omega)\coloneq\{x_a: a\in\{0,1,2\}^{\,n}\}
$ is $(n,\eta_0,\omega)$-separated and has cardinality $\#E_n(\omega)=3^n$. Thus $s(n,\eta_0,\omega,f)\ge 3^n$, which implies
$$
\limsup_{n\to\infty}\frac{1}{n}\log s(n,\eta_0,\omega,f)\ge \log 3.
$$
Hence $h_{\mathrm{top}}^{\mathrm{(r)}}(f,\omega)\ge \log 3$.
Combining the upper and lower bounds, we conclude that for every $\omega\in\Omega$,
$h_{\mathrm{top}}^{\mathrm{(r)}}(f,\omega)=\log 3$. In particular,
$h_{\mathrm{top}}^{\mathrm{(r)}}(f)=\log 3$.
\end{proof}
\end{example}

We conclude this section by demonstrating the following result about the cohomological invariance of random topological entropy.

\begin{proposition}(Cohomological invariance of random topological entropy)\label{pro:cohominventropy}
Consider random circle endomorphisms over a measure-preserving dynamical system $(\Omega, \mathcal{F}, \mathbb{P}, \sigma)$, generated by $f,g\in\mathcal{E}(\Omega)$. Suppose that there exists a measurable family of circle homeomorphisms
$h:\Omega\times \mathbb{S}^1\to \mathbb{S}^1$, where $(\omega, x)\mapsto h_\omega(x)$, and a $\sigma$-invariant full measure set $\Omega_0\subset \Omega$ such that
$$
g_\omega=h_{\sigma\omega}\circ f_\omega \circ h^{-1}_{\omega} \quad \textrm{for all }\omega\in \Omega_0.
$$
Assume that the families $\{h_\omega\}_{\omega\in \Omega_0}$ and $\{h^{-1}_\omega\}_{\omega\in \Omega_0}$ are uniformly equicontinuous on $\mathbb{S}^1$. Then, for every $\omega\in \Omega_0$,
$h^{\mathrm{(r)}}_{\mathrm{top}}(f, \omega)=h^{\mathrm{(r)}}_{\mathrm{top}}(g, \omega)$.
In particular, if the random topological entropy is almost surely constant, then
$$
h^{\mathrm{(r)}}_{\mathrm{top}}(f)=h^{\mathrm{(r)}}_{\mathrm{top}}(g).
$$
\end{proposition}
\begin{proof}
Let $\varepsilon>0$. Since $\{h^{-1}_\omega\}_{\omega\in \Omega_0}$ is uniformly equicontinuous, there exists $\delta_-(\varepsilon)>0$ such that for all $\eta\in \Omega_0$ and all $u, v\in \mathbb{S}^1$,
$$
d(u, v)<\delta_-(\varepsilon) \quad\Longrightarrow\quad d(h^{-1}_\eta(u), h^{-1}_\eta(v))<\varepsilon.
$$
We claim that every $(n, \varepsilon, \omega)$-separated set for $f$ is sent by $h_\omega$ to an $(n, \delta_-(\varepsilon), \omega)$-separated set for $g$. Let $E\subset \mathbb{S}^1$ be $(n, \varepsilon, \omega)$-separated set, and suppose that $h_\omega(E)$ were not $(n, \delta_-(\varepsilon), \omega)$-separated set. Then there would exist distinct $x, y \in E$ such that
$$
d(g^{(k)}_\omega(h_\omega(x)), g^{(k)}_\omega(h_\omega(y) ))<\delta_-(\varepsilon) \quad \textrm{for all }0\leq k<n.
$$
By the cohomological equivalence, this means
$$
d(h_{\sigma^k\omega}(f^{(k)}_\omega(x)), h_{\sigma^k\omega}(f^{(k)}_\omega(y)))<\delta_-(\varepsilon) \quad \textrm{for all }0\leq k<n.
$$
Using the uniform equicontinuity of the family $\{h^{-1}_\omega\}_{\omega\in\Omega_0}$, we obtain
$d(f^{(k)}_\omega(x), f^{(k)}_\omega(y))<\varepsilon$ for all $0\leq k<n$,
which contradicts the fact that $E$ is $(n, \varepsilon, \omega)$-separated. Therefore
$s(n, \varepsilon, \omega, f)\leq s(n, \delta_-(\varepsilon), \omega, g)$, 
which implies that $h^{\mathrm{(r)}}_{\mathrm{top}}(f, \omega)\leq h^{\mathrm{(r)}}_{\mathrm{top}}(g, \omega)$.

For the reverse inequality, let $\varepsilon>0$. Since the family $\{h_\omega\}_{\omega\in \Omega_0}$ is uniformly equicontinuous, there exists $\delta_+(\varepsilon)>0$ such that for all $\eta\in \Omega_0$ and all $u, v\in \mathbb{S}^1$,
$$
d(u, v)<\delta_+(\varepsilon) \quad\Longrightarrow\quad d(h_\eta(u), h_\eta(v))<\varepsilon.
$$
By a similar argument, every $(n, \varepsilon, \omega)$-separated set for $g$ pulls back under $h^{-1}_\omega$ to an $(n, \delta_+(\varepsilon), \omega)$-separated set for $f$. Hence $s(n, \varepsilon, \omega, g)\leq s(n, \delta_+(\varepsilon), \omega, f)$.
Passing to the limit yields $h^{\mathrm{(r)}}_{\mathrm{top}}(g, \omega)\leq h^{\mathrm{(r)}}_{\mathrm{top}}(f, \omega)$.
Combining the two equalities, we have $h^{\mathrm{(r)}}_{\mathrm{top}}(f, \omega)= h^{\mathrm{(r)}}_{\mathrm{top}}(g, \omega)$ for every $\omega\in\Omega_0$.
\end{proof}

\section{Conclusions and open problems}
In this paper, we established an interval description of the random rotation set for integrable random circle endomorphisms over an ergodic base transformation. More precisely, we showed that the random rotation set is almost surely the compact interval whose endpoints are the mean random rotation numbers of the lower and upper envelopes. We also proved two realisation results: that every value in this interval is realised as an asymptotic average displacement, and that every closed subinterval is realised as the set of accumulation points of the displacement averages along a single orbit.

The proofs combine the random rotation theory for random monotone maps of the circle, the interpolating family construction, the coincidence set mechanism for transferring orbit information from the monotone system to the original random circle endomorphism, and a full-cover construction yielding positive entropy when the random rotation interval has positive length.

\medskip

Several natural questions remain open. The first concerns whether the orbit realisation results can be strengthened to invariant random sets. In the uniquely ergodically-forced setting, Glendinning, J\"ager and Stark~\cite{Glendinningetal2009} proved that each rotation value is realised on a minimal set, and that the dynamics restricted to this minimal set has zero topological entropy. Our realisation theorem is instead an orbit-level statement on a full-measure set. This leads to the following question.

\begin{Q1}
For each value $c\in [\rho(F_L), \rho(F_U)]$, does there exist a random compact subset $K_c$ of $\mathbb{S}^1$ satisfying $f_\omega(K_c(\omega))=K_c(\sigma \omega)$ for almost every $\omega\in\Omega$ such that for every $x\in \pi^{-1}(K_c(\omega))$,
$$
\lim_{n\to \infty}\frac{F^{(n)}_\omega(x)-x}{n}=c \quad \textrm{for a.e. }\omega\in \Omega\,?
$$ 
Can $K_c$ be chosen to be minimal in an appropriate random sense? Is the restricted random topological entropy on $K_c$ zero?
\end{Q1}

The second question concerns the entropy mechanism. In the proof of Theorem \ref{thm:positivetopentropy}, the argument does not produce a canonical random symbolic subsystem, nor does it give a lower bound depending only on the length of the random rotation set.

\begin{Q2}
Under what additional assumptions can the entropy result be upgraded to a random horseshoe or a symbolic random subsystem? More precisely, if $\rho(F_U)>\rho(F_L)$, does there exist, perhaps after passing to an induced random system or a finite cover, a random compact invariant set on which the dynamics factors onto a non-trivial random subshift? Can one obtain quantitative lower bounds for $h_{\mathrm{top}}^{(\mathrm{r})}(f)$ in terms of recurrence data of the full-cover blocks, rather than merely proving positivity?
\end{Q2}

\printbibliography
\end{document}